\documentclass[reqno,12pt,letterpaper]{amsart}
\usepackage{amsmath,amssymb,mathrsfs,graphicx,amsthm, mathtools}
\usepackage[dvipsnames]{xcolor}
\usepackage[colorlinks=true,linkcolor=Red,citecolor=Green]{hyperref}
\usepackage[T1]{fontenc}
\usepackage[utf8]{inputenc}
\usepackage{tikz-cd}

\renewcommand{\baselinestretch}{1.10}

\newtheorem{proposition}{Proposition}[section]
\newtheorem{theorem}[proposition]{Theorem}
\newtheorem{lemma}[proposition]{Lemma}
\newtheorem{corollary}[proposition]{Corollary}
\theoremstyle{definition}

\theoremstyle{remark}
\newtheorem{remark}[proposition]{Remark}
\theoremstyle{definition}

\numberwithin{equation}{section}

\newcommand{\N}{\mathbb{N}}
\newcommand{\R}{\mathbb{R}}

\usepackage[labelfont=up]{subcaption}
\usepackage{comment}
\title[Thermostats without conjugate points]{Thermostats without conjugate points}

\author[J.~Echevarría Cuesta]{Javier Echevarría Cuesta}
\author[J.~Marshall Reber]{James Marshall Reber}
\address{Department of Pure Mathematics and Mathematical Statistics, University of Cambridge, Cambridge CB3 0WB, UK}
\email{je396@cam.ac.uk}

\address{Department of Mathematics, University of Chicago, Chicago, IL 60637, USA}
\email{jmarshallreber@uchicago.edu}

\begin{document}

\begin{abstract}
{We generalize Hopf’s theorem to thermostats: the total thermostat curvature of
a thermostat without conjugate points is non-positive and vanishes only if the thermostat
curvature is identically zero. We further show that, if the thermostat curvature is zero, then
the flow has no conjugate points and the Green bundles collapse almost everywhere. Given
a thermostat without conjugate points, we prove that the Green bundles are transverse
everywhere if and only if it is projectively Anosov. Finally, we provide an example showing
that Hopf’s rigidity theorem on the $2$-torus cannot be extended to thermostats. It is also the
first example of a projectively Anosov thermostat which is not Anosov.}
\end{abstract}

\maketitle

\section{Introduction}

Thermostats model the motion of a particle moving on a surface under the influence of
a force acting orthogonally to the velocity. Unlike the special case of magnetic flows,
these systems allow the force to depend on the particle’s velocity, yielding examples
of dissipative flows that still preserve the initial kinetic energy. As such, they provide
interesting models in non-equilibrium statistical mechanics, as studied by Gallavotti and
Ruelle in \cite{gallavotti97, gallavotti99, ruelle99}.

Concretely, let $(M,g)$ be a closed oriented Riemannian surface and let $\lambda \in \mathcal{C}^\infty(SM, \R)$ be a smooth function on the unit tangent bundle $\pi : SM\to M$. A curve $\gamma : \R \rightarrow M$ is a \emph{thermostat geodesic} if it satisfies the second-order differential equation
\begin{equation}
    \label{eq:defn_generalized_therm_geo} \nabla_{\dot{\gamma}} \dot{\gamma} = \lambda(\gamma, \dot{\gamma}) J\dot{\gamma},
\end{equation}
where $\nabla$ is the Levi-Civita connection associated to $g$ and $J : TM \rightarrow TM$ is the complex structure on $M$ induced by the orientation, that is, rotation by $\pi/2$ according to the orientation of the surface. Since the speed of $\gamma$ remains constant, this equation determines a flow on $SM$ given by $\varphi_t(\gamma(0), \dot{\gamma}(0)) \coloneqq (\gamma(t), \dot{\gamma}(t))$. Its infinitesimal generator is $F \coloneqq X + \lambda V$, where $X$ is the geodesic vector field and $V$ is the vertical vector field (see \cite[Lemma 7.4]{merry11}). The triple $(M,g,\lambda)$ is called a \emph{thermostat}.

The degree of freedom that comes from the choice of $\lambda$ enables thermostats to encode
a wide range of dynamical systems. Moreover, the specific dependence of $\lambda$ on velocity
can have drastic effects on key dynamical properties of the flow such as the following.

%Gaussian thermostats, i.e., when $\lambda$ depends linearly on the velocity, provide interesting
%models in non-equilibrium statistical mechanics, as first explained by Gavallotti and Ruelle in \cite{gallavotti97, gallavotti99, ruelle99}. Even beyond that, generalized thermostats encode a wide range of interesting dynamical systems, each with different key properties. For instance:

\begin{itemize}
\item Geodesic flows ($\lambda=0$). These are contact, volume-preserving, and reversible in the sense that the flip map $(x,v)\mapsto (x,-v)$ on $SM$ conjugates $\varphi_t$ with $\varphi_{-t}$. 
\item Magnetic flows ($\lambda$ depends only on position). These are still volume-preserving, but become irreversible if $\lambda \neq 0$. Since \eqref{eq:defn_generalized_therm_geo} is no longer homogeneous, the dynamics can change drastically based on the kinetic energy. Indeed, magnetic geodesics of different speeds are not just reparameterizations of unit-speed magnetic geodesics, leading to the rich theory of Mañé's critical values \cite{mane97, contreras97, burns02, contreras04, cieliebak10}.
\item Gaussian thermostats ($\lambda$ depends linearly on velocity). These are reversible, but may not preserve any absolutely continuous measure \cite{dairbekov07}. Originally introduced in non-equilibrium statistical mechanics (sometimes under the label of \textit{isokinetic dynamics}) based on Gauss's principle of least constraint \cite{hoover86, gallavotti97, gallavotti99, ruelle99}, they were later recognized as specific reparameterizations of geodesic flows of Weyl connections \cite{wojtkowski00b, wojtkowski02} and, subsequently, as the geodesic flows of metric connections, including those with non-zero torsion \cite{przytycki08} (see also \cite[Lemma 7.12]{merry11}).
\item Quasi-Fuchsian flows ($\lambda$ is the real part of a holomorphic quadratic differential). When these are Anosov, the weak stable and unstable subbundles are smooth \cite{ghys92, paternain07, cekic26}, yet they are not volume-preserving for $\lambda\neq 0$, so they cannot be algebraic. Any Anosov flow on a closed 3-dimensional manifold with smooth weak stable and unstable subbundles is smoothly orbit equivalent to either the suspension of a diﬀeomorphism of the $2$-torus or, up to finite covers, a quasi-Fuchsian flow  \cite[Theorem 4.6]{ghys93}.
\item Projective flows ($\lambda$ has linear and cubic velocity terms). These are the geodesic flows of torsion-free affine connections, up to reparameterization \cite{labourie07, mettler20}.
\item Coupled vortex equations ($\lambda$ is the real part of a holomorphic differential of degree $m\geq 2$). As explained in \cite[\S5]{dumas15}, these systems are a variation of the Abelian vortex equations on Riemann surfaces from gauge theory. These equations were introduced in the Ginzburg--Landau model for superconductors \cite{ginzburg50} before being generalized and extensively studied in relation to Yang--Mills--Higgs theory \cite{jaffe80, noguchi87, bradlow91, garcia-prada94, witten07}. For $m=3$, we obtain Wang's equation, which arises in the study of aﬃne spheres \cite{wang91}. 
\end{itemize}

As in the Riemannian setting, we define the exponential map $\exp^\lambda:TM\to M$ by
\begin{equation}\label{eq:exponential-map}
    \exp_x^\lambda(t v)\coloneqq\pi(\varphi_t(x,v)), \qquad x\in M, \, t\geq 0,\, v\in S_xM.
\end{equation}
For every $x\in M$, the map $\exp_x^\lambda$ is $\mathcal{C}^\infty$ on $T_xM\setminus \{0\}$ but, in general, only $\mathcal{C}^1$ at $0$; see, for instance, the proof of \cite[Lemma A.7]{dairbekov07b}. The lack of smoothness at the origin reflects the potential non-reversibility of thermostat flows. We say that the thermostat in question has \emph{no conjugate points} if $\exp_x^\lambda$ is a local diffeomorphism for all $x\in M$. Given a thermostat geodesic segment $\gamma: [0,T]\to M$ with distinct endpoints $x_0 \coloneqq \gamma(0)$ and $x_1\coloneqq\gamma(T)$, we say that $x_0$ and $x_1$ are \textit{conjugate along} $\gamma$ if the map $\exp_{x_0}^\lambda$ is singular at $T\dot{\gamma}(0)$, that is, if the differential $d_{T\dot{\gamma}(0)}\exp_{x_0}^\lambda$ has a non-trivial kernel. The goal of this paper is to explain how the no-conjugate-points condition relates to different notions of curvature, as well as characterize the dynamics of such thermostats. In doing so, we highlight both the concepts that generalize perfectly from the geodesic case and the nuances that appear with greater dynamical complexity. This exercise not only sheds new light on thermostats, but also gives new results for geodesic flows.

%This paper aims to explain how 
%certain rigidity results for geodesic flows without conjugate points extend to thermostats, highlighting both the concepts that generalize perfectly and the nuances that appear with greater dynamical complexity. 

\subsection{Conjugate points and thermostat curvature} 

Let $K_g\in \mathcal{C}^\infty(M, \R)$ denote the Gaussian curvature of $(M,g)$. In the geodesic case, it is easy to check that $K_g \leq 0$ implies $g$ has no conjugate points. The quantity that usually plays the role of Gaussian curvature for thermostats is the \textit{thermostat curvature} $\mathbb{K}\in \mathcal{C}^\infty(M, \R)$ given by
$$\mathbb{K}\coloneqq\pi^*K_g - H\lambda+\lambda^2+FV\lambda,$$
where $H\coloneqq[V,X]$. This is a smooth function on $SM$ instead of $M$. 
It turns out that this notion of thermostat curvature is implicitly making a choice of a gauge (see \S \ref{subsection:characteristic}). Given a function $p: SM\to \R$ that is $\mathcal{C}^1$ along the flow, define
\begin{equation}\label{eq:definition-of-kappa-p}
    \kappa_p \coloneqq \pi^*K_g -H\lambda +\lambda^2 + Fp+p(p-V\lambda). 
\end{equation}
Observe that $\mathbb{K}$ corresponds to the special case $p=V\lambda$. The quantity $\kappa_p$ was explicitly  introduced in \cite{mettler19} as a tool for analyzing the dynamics of thermostat flows, but it already appears implicitly in \cite{jane09}.

Our first result shows that the notion of curvature $\kappa_p$ offers a useful criterion to check whether a thermostat has no conjugate points.

\begin{theorem}\label{theorem:non-positive-implies-no-conjugate}
Let $(M, g, \lambda)$ be a thermostat. If $\kappa_p\leq 0$ for some function $p:SM\to \R$ that is $\mathcal{C}^1$ along the flow, then there are no conjugate points.
\end{theorem}

In fact, the additional degree of freedom that comes from the gauge $p$ allows us to completely characterize thermostats without conjugate points. This characterization is also new for geodesic flows. 

\begin{theorem}\label{theorem:no-conjugate-points-characterization}
A thermostat $(M, g, \lambda)$ has no conjugate points if and only if there exists a Borel measurable function $p: SM\to \R$ smooth along the flow with $\kappa_p=0$.
\end{theorem}

Next, we give a generalization of Hopf's rigidity result in \cite{hopf48} to thermostats. Note that $\mu$ denotes the Liouville form on $SM$.
%This was done for Gaussian thermostats in \cite{assylbekov14} .

\begin{theorem}\label{theorem:hopf}
    Let $(M, g, \lambda)$ be a thermostat without conjugate points. For any Borel measurable function $p: SM\to \R$ that is $\mathcal{C}^1$ along the flow, we have
    \begin{equation}\label{eq:hopf-inequality}
        \int_{SM}\kappa_p \mu\leq \int_{SM}(p-V\lambda)^2 \mu.
    \end{equation}
    Moreover, if equality holds, then $\mathbb{K}=0$. 
    %Moreover, equality holds if and only if $\mathbb{K}=0$.
\end{theorem}
As a consequence of the inequality \eqref{eq:hopf-inequality} with $p=V\lambda$, Stokes's theorem, and the Gauss--Bonnet theorem, we get
\begin{equation}\label{eq:hopf-inequality-2}
    2\pi \chi(M)+\int_{SM}(\lambda^2-(V\lambda)^2) \mu\leq 0,
\end{equation}
where $\chi(M)$ is the Euler characteristic of $M$. Furthermore, observe that the exponential map defined in \eqref{eq:exponential-map} cannot yield a covering of the $2$-sphere $\mathbb{S}^2$ for topological reasons, so any surface admitting a thermostat without conjugate points must have genus at least one. 

Let us briefly focus on the case where $M $ is homeomorphic to the $2$-torus $\mathbb{T}^2$. In the geodesic case, we recover the  classical fact by Hopf that $g$ must be flat. In the magnetic case, the inequality \eqref{eq:hopf-inequality-2} implies that $\lambda = 0$ and then $g$ must also be flat (see also \cite[Corollary C]{assenza25}). In some sense, these observations are telling us that the situation on $\mathbb{T}^2$ is very rigid when the flow is volume-preserving. If we allow $\lambda$ to also have a linear term with respect to velocity, \cite[Theorem 1.1]{assylbekov14} tells us that the magnetic component  of $\lambda$ (that is, the zeroth Fourier mode $\lambda_0$) must still identically vanish and the metric $g$ must be conformally flat. 
The following no-go result shows that this rigidity does not apply to more general thermostats on $\mathbb{T}^2$.

\begin{theorem}\label{theorem:counterexample}
    For any Riemannian metric $g$ on $\mathbb{T}^2$, there exists $\lambda\in \mathcal{C}^\infty(S\mathbb{T}^2, \R)$ such that the thermostat $(\mathbb{T}^2, g, \lambda)$ has no conjugate points. Moreover, the function $\lambda$ can always be chosen such that $\lambda_0\neq 0$. 
\end{theorem}

\subsection{Green bundles in the cotangent bundle} 

A key observation in the study of metrics without conjugate points is the existence of two flow-invariant subbundles of $T(SM)$, known as Green bundles. The construction of these subbundles was extended to the setting of convex Hamiltonians in \cite{contreras99}. However, these arguments do not directly carry over to the thermostat setting, as thermostats may be dissipative.

We take a new approach to understanding the Green bundles by working on the cotangent bundle $T^*(SM)$ as opposed to the tangent bundle $T(SM)$. The primary motivation is that, if we consider the induced dynamics, there is a smooth invariant subbundle $\Sigma\subset T^*(SM)$ which, in spirit, can replace the notion of a contact distribution in the geodesic case. Indeed, the symplectic lift of $\varphi_t$ to $T^*(SM)$, given by
\begin{equation}\label{eq:defn_symplectic_lift}
    \tilde{\varphi}_t(v, \xi) \coloneqq\left(\varphi_t(v), d_v\varphi_t^{-\top} (\xi)\right), \qquad (v,\xi)\in T^*(SM),
\end{equation}
is the Hamiltonian flow of $\xi(F(v))$, so it preserves the \emph{characteristic set} $\Sigma$ with fibers 
\begin{equation} \label{eq:defn_characteristic}
    \Sigma(v) \coloneqq \{\xi \in T_v^*(SM) \ | \ \xi(F(v)) = 0\}.
\end{equation}

When working on $T^*(SM)$, it is natural to introduce a moving coframe. The vector fields $(X, H, V)$ form an orthonormal frame for $T(SM)$ with respect to the Sasaki metric (the natural lift of $g$ to $SM$). We can then consider the corresponding dual frame $(\alpha, \beta, \psi)$ for $T^*(SM)$. The cohorizontal subbundle is defined as $\mathbb{H}^*\coloneqq\R\beta$, whereas the covertical is $\mathbb{V}^*\coloneqq\R\psi$. In the geodesic case, we have $\Sigma=\mathbb{H}^*\oplus\mathbb{V}^*$. For a thermostat, we introduce $\psi_\lambda\coloneqq\psi-\lambda \alpha$ and $\mathbb{V}_\lambda^*\coloneqq\R\psi_\lambda$ so that $\Sigma = \mathbb{H}^*\oplus \mathbb{V}^*_\lambda$. See Figure \ref{figure:subbundles}.

\begin{figure}[h]
\begin{center}
\includegraphics[scale=0.40]{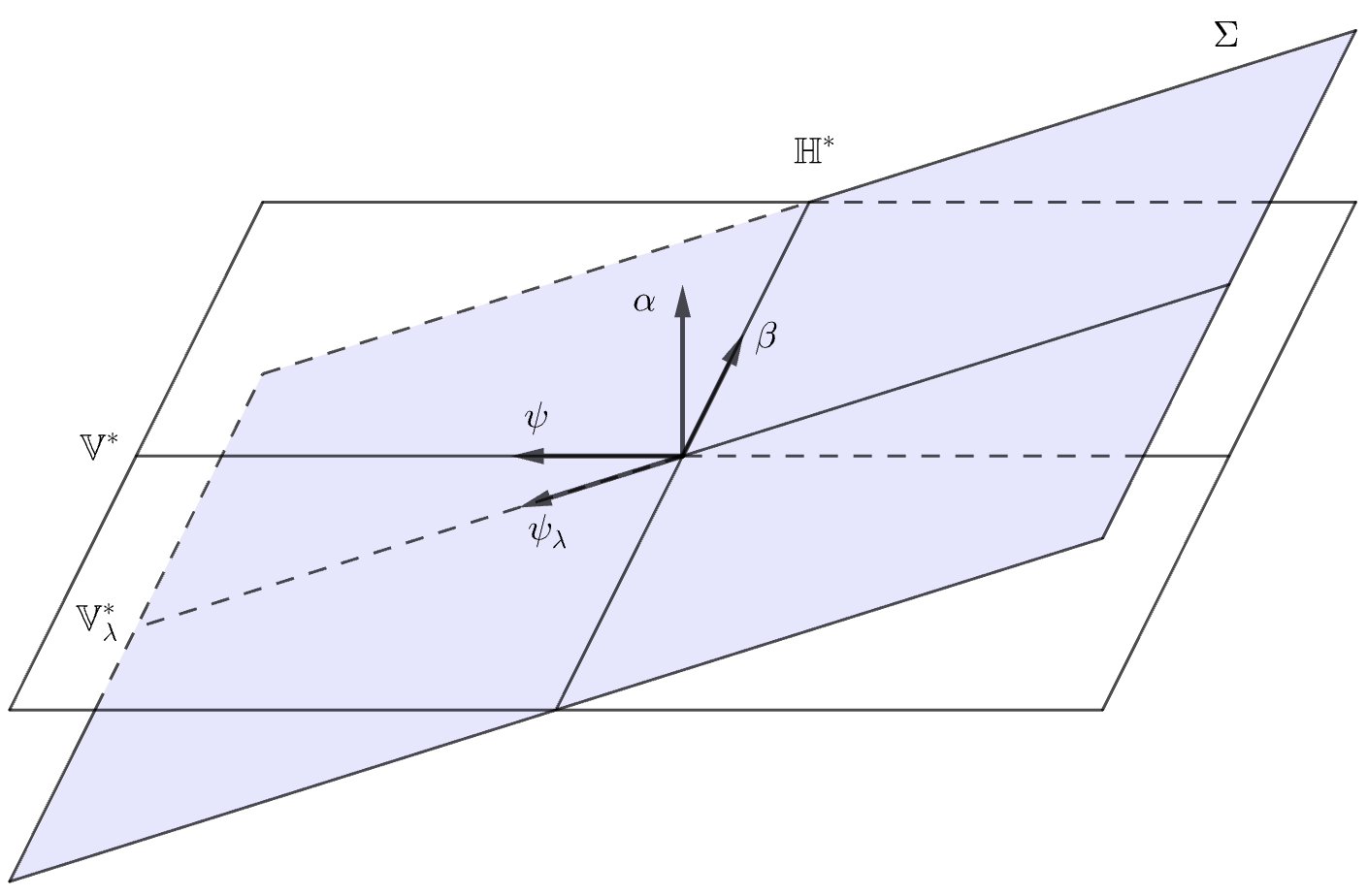}
\end{center}
\caption{The adapted frame $(\alpha, \beta, \psi_\lambda)$ for $T^*(SM)$.}
\label{figure:subbundles}
\end{figure}

%A first step is to construct the \textit{Green bundles} on orbits without conjugate points. 
In this new setting, we construct the \emph{Green bundles} on orbits without conjugate points.
Note that these are usually defined as subbundles of $T(SM)$ instead of $T^*(SM)$.

\begin{theorem}\label{theorem:green-bundles-definition}
Let $(M, g, \lambda)$ be a thermostat. For a given $v\in SM$, either of the following are observed.
    \begin{enumerate}
        \item[(a)] There exist $t_0\in \R$ and $t\neq 0$ such that
        \begin{equation}
            d\varphi_{t}^{-\top}(\mathbb{H}^*(\varphi_{t_0}(v))=\mathbb{H}^*(\varphi_{t_0+t}(v)).
        \end{equation}
        This happens if and only if the points $\pi(\varphi_{t_0}(v))$ and $\pi(\varphi_{t_0+t}(v))$ on $M$ are conjugate along the thermostat geodesic $t\mapsto \pi(\varphi_t(v))$.
        \item[(b)] There exist two invariant subbundles $G_{s/u}^*\subset \Sigma$ along the orbit of $v$ given by
        \begin{equation}\label{eq:green-bundles-definition}
            \begin{alignedat}{1}
G^*_s(v):&=\lim_{t\to \infty}d\varphi_{-t}^{-\top}(\mathbb{H}^*(\varphi_{t}(v)),\\
 G^*_u(v):&=\lim_{t\to \infty}d\varphi_{t}^{-\top}(\mathbb{H}^*(\varphi_{-t}(v)).
            \end{alignedat}
        \end{equation}
    %$$G^*_s(v)\coloneqq\lim_{t\to +\infty}d\varphi_{-t}^{-\top}(\mathbb{H}^*(\varphi_{t}(v)), \qquad G^*_u(v)\coloneqq\lim_{t\to -\infty}d\varphi_{-t}^{-\top}(\mathbb{H}^*(\varphi_{t}(v)).$$
    They satisfy the transversality condition
    \begin{equation}\label{eq:h-transversality}
        G^*_s(v)\cap \mathbb{H}^*(v)=\{0\}=G^*_u(v)\cap \mathbb{H}^*(v).
    \end{equation}
\end{enumerate}
   % $G^s(v)=G^u(v)=\mathbb{V}_\lambda^*$ if and only if $\kappa(\varphi_t(v))=0$ for all $t\in \R$. 
    
    %, and $G^s(v)\cap G^u(v)=\{0\}$ if and only if the orbit $t\mapsto \varphi_t(v)$ on $SM$ is Anosov.
\end{theorem}

In \cite[Theorem 3.2]{eberlein73}, Eberlein characterized Anosov geodesic flows as the geodesic flows without conjugate points that have transverse Green bundles. In \cite[Theorem C]{contreras99}, a similar statement was given for convex Hamiltonians (which might be non-contact, but are still volume-preserving). Recall that a flow $\{\varphi_t\}_{t \in \R}$ is \emph{Anosov} if there is a flow-invariant splitting $$T(SM) = \R F \oplus E_s \oplus E_u$$ and constants $C \geq 1$ and $0 < \rho < 1$ such that, for all $v\in SM$ and $t \geq 0$, we have 
%\[ \|d_v\varphi_t|_{E_s(v)}\| \leq C \rho^t \text{ and } \|d_v\varphi_{-t}|_{E_u(v)}\| \leq C \rho^t.\] 
\begin{equation}\label{eq:hyperbolic-estimates}
\Vert d_v\varphi_{t}|_{E_s(v)}\Vert \leq C\rho^t, \qquad\Vert d_v\varphi_{-t}|_{E_u(v)}\Vert \leq C\rho^t.
\end{equation}

As we will see, due to the lack of volume preservation, the natural extension of these results to thermostats applies to projectively Anosov flows instead.  The flow $\{\varphi_t\}_{t\in \R}$ is said to be \emph{projectively Anosov} if there exist a flow-invariant splitting of the quotient tangent bundle
$$T(SM)/\R F =\mathcal{E}_s\oplus \mathcal{E}_u$$
and constants $C\geq 1$ and $0 <\rho < 1$ such that, for any $v\in SM$ and $t\geq 0$, we have
\begin{equation}\label{eq:dominated-estimates}
\Vert \left.d_v \varphi_t\right|_{\mathcal{E}_s(v)}\Vert  \Vert \left.d_{\varphi_t(v)}\varphi_{-t}\right|_{\mathcal{E}_u(\varphi_t(v))}\Vert\leq C \rho^t,
\end{equation}
where $\Vert \cdot \Vert$ denotes the operator norm induced by the Sasaki metric on $T(SM)$. We assume that both subbundles $\mathcal{E}_{s}, \mathcal{E}_u \subset T(SM)/\R F$ are non-trivial and we say that $\mathcal{E}_u$ \textit{dominates} $\mathcal{E}_s$. These subbundles lift to flow-invariant subbundles $q^{-1}(\mathcal{E}_s)$ and $q^{-1}(\mathcal{E}_u)$ of $T(SM)$, where $q: T(SM)\to T(SM)/\R F$ is the quotient map. We refer to them as the \textit{weak stable} and \textit{unstable subbundles}, respectively.

Projectively Anosov flows also appear in the literature as \textit{conformally Anosov} flows \cite{eliashberg98, blair98} or as flows admitting a \textit{dominated splitting} \cite{arroyo03}. In the latter paper, the authors explain why this definition is the `adequate' notion of dominated splitting for flows. This property does not appear in geodesic or magnetic settings because volume-preserving projectively Anosov flows are Anosov \cite[Proposition 2.34]{araujo10}. In those cases, there is no need to differentiate between the notions. It is also worth noting that one could define the Anosov property directly on the quotient tangent bundle by requiring that the subbundles $\mathcal{E}_{s/u}$ satisfy the stronger estimates \eqref{eq:hyperbolic-estimates} instead of \eqref{eq:dominated-estimates}. These points of view are equivalent: in one direction, it suffices to take $\mathcal{E}_{s/u}\coloneqq q(E_{s/u})$ and the converse is given by \cite[Proposition 5.1]{wojtkowski00}.

The thermostat version of Eberlein's result is hence the following theorem.

\begin{theorem}\label{theorem:characterization-dominated-splitting}
Let  $(M, g, \lambda)$ be a thermostat without conjugate points. It is projectively Anosov if and only if $G^*_s(v)\cap G^*_u(v)=\{0\}$ for all $v\in SM$.
\end{theorem}

\begin{remark} \label{rem:klingenberg}
    In establishing the result for geodesic flows, Eberlein uses Klingenberg \cite{klingenberg74} to argue that an Anosov geodesic flow must have no conjugate points. In the absence of such a result for projectively Anosov thermostats, we assume that the thermostat is without conjugate points.
\end{remark}

For a projectively Anosov thermostat, we no longer have a direct sum decomposition of the tangent bundle, but only the weaker relation
$$T(SM)=q^{-1}(\mathcal{E}_s)+ q^{-1}(\mathcal{E}_u)$$
with $q^{-1}(\mathcal{E}_s)\cap q^{-1}(\mathcal{E}_u)=\R F$. Nonetheless, we can still define the \emph{dual stable} and \emph{unstable subbundles} $E^*_{s}, E^*_u\subset T^*(SM)$ via
\begin{equation}\label{eq:definition-cotangent-bundles}
   E^*_s (q^{-1}(\mathcal{E}_{s}))=0=E^*_u(q^{-1}(\mathcal{E}_u)).
\end{equation}
With this definition, we then obtain the direct sum decomposition 
$$\Sigma = E^*_s \oplus E^*_u,$$
and $E^*_{s/u}$ satisfy analogues of the estimates \eqref{eq:dominated-estimates}, with $d\varphi_t$ replaced by $d\varphi_t^{-\top}$. If the thermostat has no conjugate points and is projectively Anosov, then $G^*_{s/u}=E_{s/u}^*$.

To prove Theorem \ref{theorem:characterization-dominated-splitting}, we use the following characterization of thermostats with transverse Green bundles. Interestingly, even
in the well-studied geodesic case, it provides a new partial converse to \cite{anosov67}: although the Anosov property does not imply negative Gaussian curvature \cite{donnay03, donnay18}, it does tell us that the thermostat curvature with respect to an appropriate gauge is negative everywhere.

\begin{theorem}\label{theorem:partial-converse}
    If a thermostat $(M,g,\lambda)$ without conjugate points has transverse Green bundles, there is a continuous function $p : SM \to \R$ smooth along the flow with $\kappa_p < 0$.
\end{theorem}

Finally, we explore
the other extreme; namely, when the Green bundles collapse to a line everywhere instead of being transverse. For geodesic flows without conjugate points, a conjecture of Freire and Mañé can be rephrased as stating that the Green bundles collapse if and only if the metric is flat \cite{freire82}. Note that, in the geodesic case, the Green bundles collapsing implies that the fundamental group grows sub-exponentially. However then, if $M$ is a surface, it must be the $2$-sphere or the $2$-torus. The conjecture for surfaces then follows by \cite{hopf48}. 

A natural question is whether this extends to the thermostat setting with the thermostat curvature in place of the Gaussian curvature. As we will see in Proposition \ref{propn:example}, this is not the case. In spite of this, there is still a connection between the Green bundles collapsing to a line, the thermostat curvature $\mathbb{K}$, and a quantity which we refer to as the \emph{damped thermostat curvature},
\begin{equation}\label{eq:definition-of-damped-thermostat-curvature}
    \tilde{\kappa} \coloneqq \pi^*K_g - H\lambda + \lambda^2 + \frac{FV\lambda}{2} - \frac{(V\lambda)^2}{4}.
\end{equation}
Note that $\tilde{\kappa}$ is simply $\kappa_p$ with the particular choice of $p=V\lambda/2$.

\begin{theorem} \label{theorem:collapse}
    Let $(M, g, \lambda)$ be a thermostat.
    \begin{enumerate}
        \item[(a)] If $\mathbb{K} = 0$, then the flow has no conjugate points and $G^*_s=G^*_u = \R(\psi_\lambda - V(\lambda) \beta)$ $\mu$-almost everywhere. Moreover, if
$V\lambda=0$, then $G^*_s=G^*_u = \R\psi_\lambda$ everywhere. 
        
        \item[(b)] If $\tilde{\kappa} \leq 0$, then, for any invariant Borel measure $\nu$ on $SM$, we have $\tilde{\kappa}= 0$ $\nu$-almost everywhere if and only if $ G^*_s=G^*_u$ $\nu$-almost everywhere.

        %\item[(c)] If $V(\lambda)=0$, then $\mathbb{K}= 0$ implies that the flow has no conjugate points and $ G^*_s=G^*_u = \R\psi_\lambda$ everywhere. 
    \end{enumerate}
    
\end{theorem}

In \S\ref{section:examples}, we will show that Theorem \ref{theorem:collapse}(a) is optimal: when $V\lambda\neq 0$, it is possible to have $\mathbb{K}=0$ yet $G^*_s(v)\cap G^*_u(v)=\{0\}$ for some $v\in SM$. See Figure \ref{figure:three-cases}. In particular, this implies that the conjecture of Freire and Mañé does \emph{not} extend to the setting of thermostats with the thermostat curvature in place of the Gaussian curvature. However, observe that if $V\lambda = 0$ (that is, the system is magnetic), then $\mathbb{K} = \tilde{\kappa} = \kappa_0$ and $\mu$ becomes an invariant measure for the thermostat flow. Thus, Theorem \ref{theorem:collapse} implies the following result.

\begin{corollary} \label{cor:collapse}
    Let $(M,g,\lambda)$ be a magnetic system with $\kappa_0 \leq 0$. We have $\kappa_0 = 0$ if and only if $G_s^* = G_u^*$ everywhere.
\end{corollary}

\begin{figure*}[h!]
\centering
\begin{subfigure}[t]{0.34\textwidth}
    \centering
\includegraphics[scale=0.24]{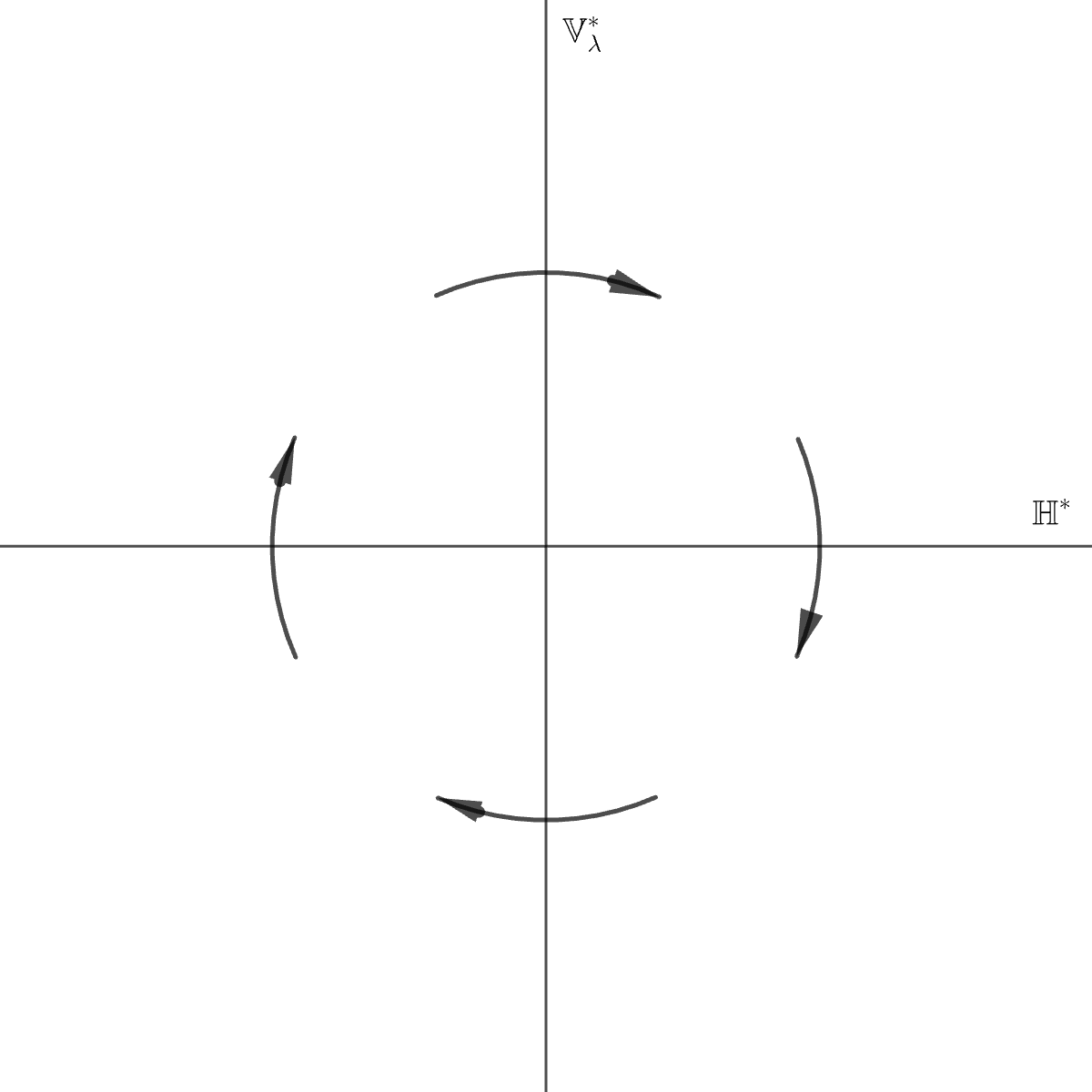}
\caption*{(a) Conjugate points}
\end{subfigure}%
\begin{subfigure}[t]{0.34\textwidth}
    \centering
\includegraphics[scale=0.24657534246575344]{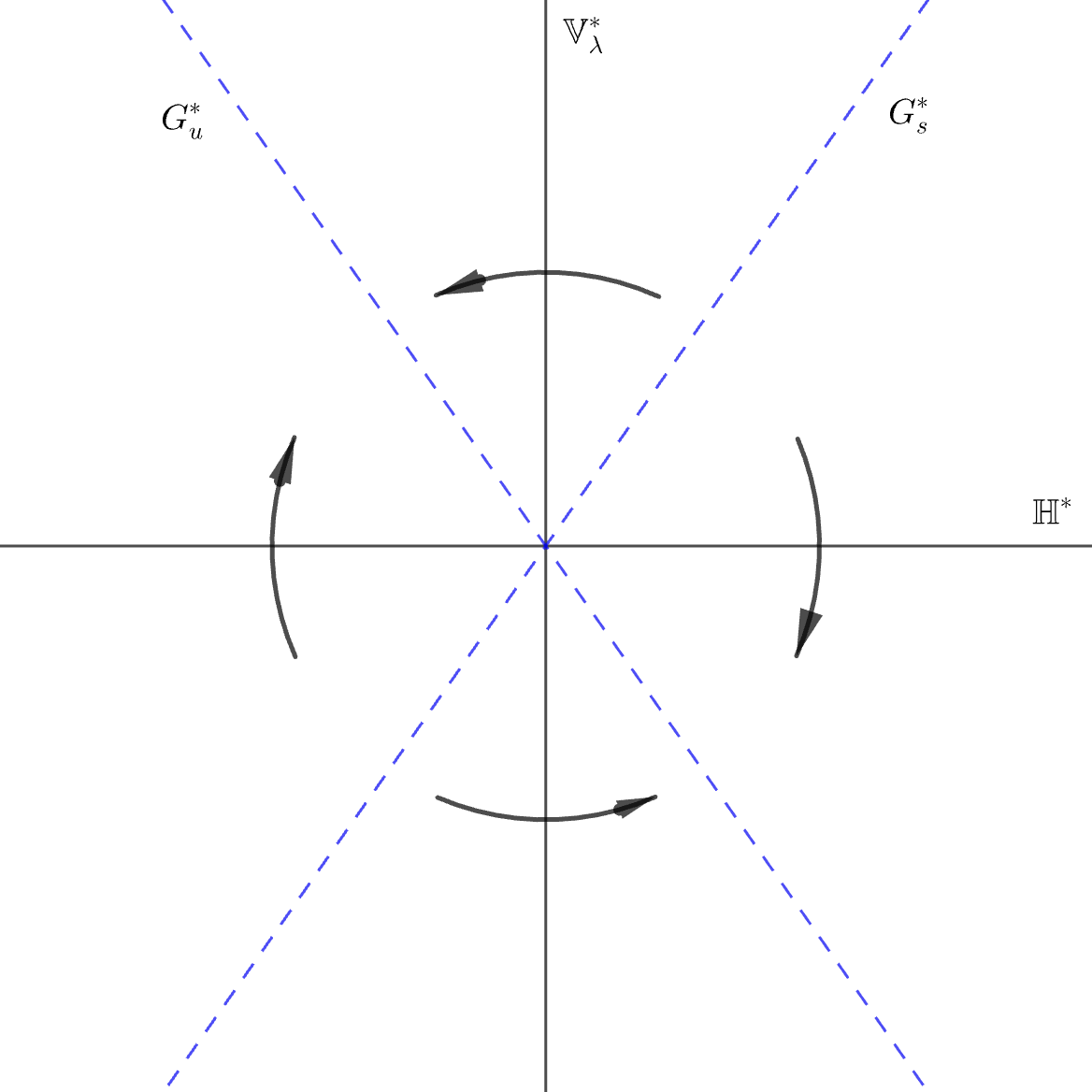}
\caption*{(b) Projectively Anosov}
\end{subfigure}%
\begin{subfigure}[t]{0.34\textwidth}
    \centering
\includegraphics[scale=0.24]{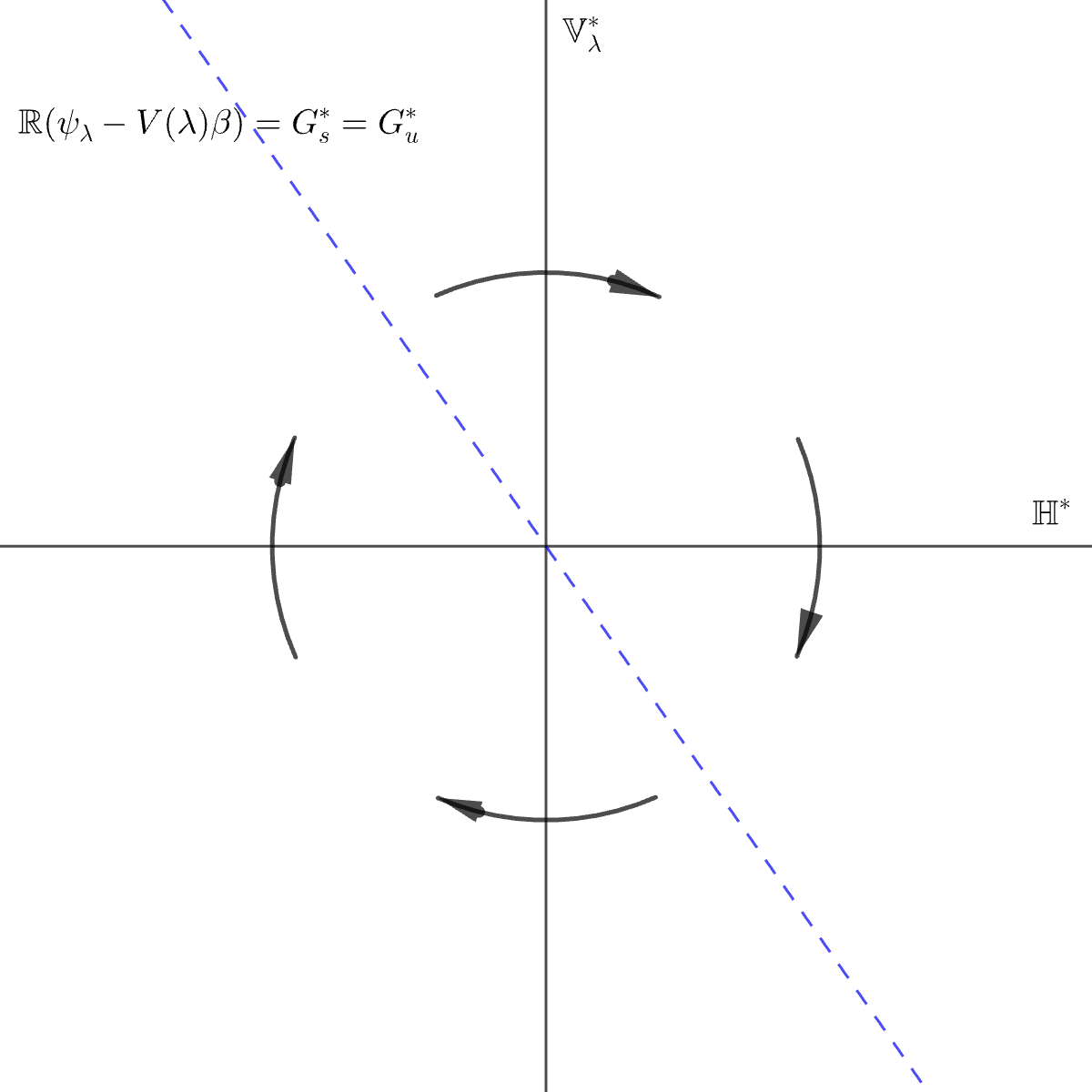}
\caption*{(c) $\mathbb{K}=0$}
\end{subfigure}
\caption{The lifted dynamics on the characteristic set $\Sigma$.}
\label{figure:three-cases}
\end{figure*}

Combining the above results with \cite{dairbekov07}[Lemma 4.1] and \cite[Proposition 3.5, Theorem 3.7]{mettler19}, we have the following picture.
\begin{center}
\small
\begin{tikzcd}
\parbox[c]{5cm}{\centering \renewcommand{\baselinestretch}{1}\selectfont $\mathbb{K}<0$,  $\kappa_0<0$, or\\ $\kappa_p + \frac{(V\lambda)^2}{4} <0$ for some $p$}
\arrow[d, Rightarrow]  &  & \kappa_p \leq 0 \text{ for some $p$} \arrow[d, Rightarrow]\\
\text{Anosov}\arrow[r, Rightarrow] \arrow[d, Rightarrow]&  \parbox[c]{5cm}{\centering \renewcommand{\baselinestretch}{1}\selectfont no conjugate points + \\ transverse Green bundles} \arrow[r, Rightarrow] \arrow[d, Rightarrow] & \text{no conjugate points} \arrow[d, Leftrightarrow]\\ 
\text{ projectively Anosov}  & \arrow[l, Rightarrow]  \kappa_p < 0 \text{ for some $p$}   & \kappa_p = 0\text{ for some $p$}\\
\end{tikzcd}
\end{center}

This should be contrasted with the following diagram, which summarizes what was previously known in the geodesic setting.
\begin{center}
\begin{tikzcd}\label{eq:first-diagram}
K_g <0 \arrow[d, Rightarrow] &  &  K_g\leq 0 \arrow[d, Rightarrow]\\
\text{Anosov}\arrow[r, Leftrightarrow]  & \parbox[c]{5cm}{\centering \renewcommand{\baselinestretch}{1}\selectfont no conjugate points + \\ transverse Green bundles}\arrow[r, Rightarrow]  &  \text{no conjugate points.} 
\end{tikzcd}
\end{center}

\subsection{Remaining questions}

As noted in Remark \ref{rem:klingenberg}, it is not clear whether one needs the assumption that the thermostat is without conjugate points in Theorem \ref{theorem:characterization-dominated-splitting}. Given an Anosov flow on $SM$, \cite[Theorem A]{ghys84} tells us that it is topologically
orbit equivalent to the geodesic flow of any metric of constant negative Gaussian curvature on $M$. Therefore, these flows are transitive and their non-wandering set is all
of $SM$. This property ends up being critical in proving that there are no conjugate
points. In contrast, as we will see with concrete examples in  \S\ref{section:examples}, projectively Anosov thermostats
can have non-trivial wandering sets. 

\noindent \textit{Question.} Can a projectively Anosov thermostat have conjugate points?

One possible approach to this problem is to try to understand thermostats on the $2$-sphere $\mathbb{S}^2$. As pointed out above, it is easy to see using the exponential map that every such thermostat must have conjugate points. If one can construct an example which is projectively Anosov, then this would show that the projectively Anosov assumption is not enough to rule out conjugate points. However, it is not even clear whether there can be an arbitrary projectively Anosov flow on the unit tangent bundle of $\mathbb{S}^2$; the work of Arroyo and Rodríguez Hertz \cite{arroyo03} gives some insight into the problem, but it does not seem to be enough to rule out such examples.

\noindent \textit{Question.} Does the unit tangent bundle of $\mathbb{S}^2$ admit a projectively Anosov flow?

Another possible approach to this problem is to try to understand whether all projectively Anosov thermostats give rise to hyperbolic behavior. In \S\ref{section:examples}, we give explicit examples of projectively Anosov thermostats on the $2$-torus which are not Anosov, showing that this is not the case, and answering a question of Mettler and Paternain \cite{mettler19} about the existence of such systems. These examples critically rely on the fact that the surface is a $2$-torus and, thus, it may be possible that a projectively Anosov thermostat will always be Anosov if the surface is not the $2$-torus. In light of the previous question, one can try exploring the following question.

\noindent \textit{Question.} Is there a projectively Anosov thermostat on a surface of genus at least two which is not Anosov?

 Next, we suspect that the assumption $\kappa_0\leq 0$ is not required in Corollary \ref{cor:collapse}, and hence, the conjecture of Freire and Mañé \emph{does} extend to the setting of magnetic systems on surfaces. Provided the function $\lambda$ is sufficiently nice (that is, if the Mañé critical value is less than $1/2$), one can use Theorem \ref{theorem:hopf} and \cite[Theorem D, Proposition 5.4]{burns02} to deduce the result. It is not immediately obvious how to show that $\kappa_0 = 0$ if one only assumes that the topological entropy is zero and the genus of $M$ is at least two.

%Observe that, by Theorem \ref{theorem:collapse} (a) and Lemma \ref{lemma:lyapunov-collapse}, one only needs to show that if $\kappa(v) \neq 0$ for some $v \in SM$, then the topological entropy of the flow is positive. In the magnetic setting, one can use Theorem \ref{theorem:hopf} and \cite[Theorem D and Proposition 5.4]{burns02} to deduce the result if the Ma\~{n}\'{e} critical value is less than $1/2$. It is unclear whether the assumption on the Ma\~{n}\'{e} critical value can be dropped, leaving us with the following.

%{\color{red} no-go Lyusternik–Fet theorem }

{\noindent \textit{Question.}} Let $(M,g,\lambda)$ be a magnetic system without conjugate points. Does $G_s^* = G_u^*$ everywhere imply $\kappa_0 = 0$?

Based on part (b) of Theorem \ref{theorem:collapse}, it is possible that the damped thermostat curvature is better equipped for detecting the topological entropy of the thermostat flow. However, even in the setting where $\tilde{\kappa} \leq 0$, it is not clear whether $\tilde{\kappa} = 0$ everywhere is equivalent to the Green bundles collapsing everywhere. 

{\noindent \textit{Question.}}
    Let $(M,g,\lambda)$ be a thermostat. Is it the case that  $\tilde{\kappa} = 0$ if and only if $G_s^*(v) = G_u^*(v)$ for all $v\in SM$?

\subsection{Organization of the paper}  In \S\ref{section:conjugate-points-and-green-bundles}, we study the relationship between conjugate points, Green bundles and thermostat curvature. We describe the lifted dynamics of the thermostat on the characteristic set in \S\ref{subsection:characteristic}. In \S\ref{subsection:cocycles},  we explain how our point of view is equivalent to the one using cocycles. The lifted dynamics give us an interpretation of the no-conjugate-points condition in terms of the twist property of the cohorizontal subbundle in \S\ref{subsection:conjugate}, unlocking Theorem \ref{theorem:non-positive-implies-no-conjugate}. Next, in \S\ref{subsection:green_bundles}, we construct the Green bundles, and prove Theorems \ref{theorem:no-conjugate-points-characterization}, \ref{theorem:hopf}, and  \ref{theorem:green-bundles-definition}. We finish the section by studying the relationship between the Lyapunov exponents of the flow and the Green bundles in \S\ref{subsec:lyapunov}, giving us the tools to prove Theorem \ref{theorem:collapse}.

In \S\ref{section:transversal_green_bundles}, we explore the relationship between the projectively Anosov property and transverse Green bundles. We first show that transverse Green bundles must be continuous and then use this property to prove Theorems \ref{theorem:characterization-dominated-splitting} and \ref{theorem:partial-converse}.

In \S\ref{section:examples}, we present a family of examples of thermostats on $\mathbb{T}^2$ without conjugate points that are projectively Anosov (yet are not Anosov), proving Theorem \ref{theorem:counterexample}. 
%To better understand them, we calculate their topological entropy in Appendix \ref{appendix:calculations} by finding their Lyapunov exponents and showing that the flows admit SRB measures. 

\section{Conjugate points and Green bundles}\label{section:conjugate-points-and-green-bundles}

In what follows, $(M, g)$ is a closed oriented Riemannian surface and we take an arbitrary $\lambda\in \mathcal{C}^\infty(SM, \R)$. %Whenever we use additional assumptions, it will be clearly stated in the result statements.

\subsection{Dynamics on the characteristic set} \label{subsection:characteristic}
Recall that $(\alpha, \beta, \psi)$ is the global coframe for $T^*(SM)$ dual to the orthonormal frame $(X, H, V)$ under the Sasaki metric. By using the commutation formulas,
 $$[V,X]=H, \qquad [V,H]=-X,\qquad [X,H]=\pi^*(K_g)V,
$$
we can derive the structure equations,
$$d\alpha = \psi \wedge \beta, \qquad d\beta=-\psi\wedge \alpha, \qquad d\psi =-\pi^*(K_g)\alpha\wedge\beta.$$

We will use the adapted coframe $(\alpha, \beta, \psi_\lambda)$. See Figure \ref{figure:subbundles}. Combining the previous structure equations with Cartan's formula, we obtain 
\begin{equation*}\label{eq:structure-equations}
\mathcal{L}_{F}\alpha =\lambda \beta, \qquad 
\mathcal{L}_{F}\beta =\psi_\lambda, \qquad
\mathcal{L}_{F}\psi_\lambda = -\kappa_0 \beta+V(\lambda)\psi_\lambda,
\end{equation*}
%$$[V,F]=H+V(\lambda)V,\qquad [V,H]=−F+\lambda V, \qquad [F,H]=-\lambda F + (\pi^*K_g-H(\lambda)+\lambda^2)V,$$
where $\kappa_0$ is defined in \eqref{eq:definition-of-kappa-p} (with $p=0$). 
%\begin{equation*}
%\kappa\coloneqq\mathbb{K}-F(V(\lambda)) = \pi^*(K_g) - H(\lambda) + \lambda^2.
%\end{equation*}

For any function $p: SM\to \R$ that is $\mathcal{C}^1$ along the flow, it will also be useful to define $\phi_p \coloneqq \psi_\lambda -p\beta$ so that $(\alpha, \beta, \phi_p)$ is an alternative coframe satisfying $\Sigma=\mathbb{H}^*\oplus \R\phi_p$ (see Figure \ref{figure:bases}). This is nothing but a change of coordinates on $\Sigma$. We now have
\begin{equation}\label{eq:structure-equations-2}
\mathcal{L}_{F}\beta =p\beta+\phi_p, \qquad
\mathcal{L}_{F}\phi_p = -\kappa_p \beta + (V\lambda - p)\phi_p.
\end{equation}
%where
%\begin{equation*} \label{eq:adapted_kappa} \kappa_p \coloneqq \kappa + F(p) + p(p - V(\lambda)).
%\end{equation*}
\begin{figure}[h]
\begin{center}
\includegraphics[scale=0.35]{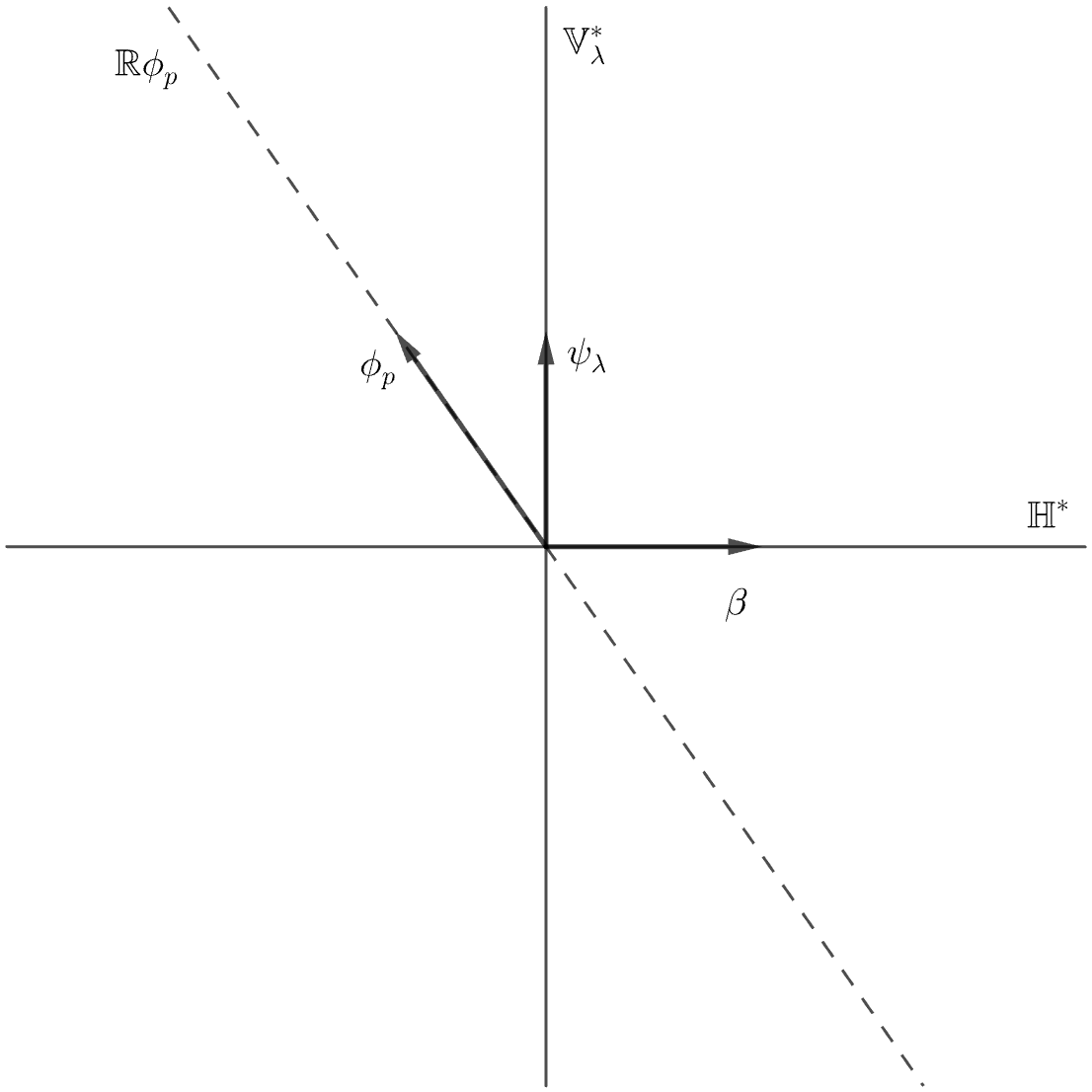}
\end{center}
\caption{The bases $(\beta, \psi_\lambda)$ and $(\beta, \phi_p)$ for $\Sigma$.}
\label{figure:bases}
\end{figure}

To each $\xi \in \Sigma(v)$, we associate the functions $x,y\in \mathcal{C}^\infty(\R, \R)$ characterized by
\begin{equation}\label{eq:definition-of-x-y}
    d_v\varphi_t^{-\top}(\xi)= x(t)\beta+y(t)\phi_p.
\end{equation}
They capture all the information of the lifted dynamics in $\Sigma$.

\begin{lemma}\label{lemma:key-lemma}
Let $\xi\in \Sigma(v)$. Along the orbit of $v$, we have the pair of equations
\begin{equation}\label{eq:flow-on-sigma}
    \begin{cases}
\dot{x}+px-\kappa_p y=0,\\
\dot{y}+x+(V\lambda-p)y = 0.
    \end{cases}
\end{equation}
In particular, the $y$ component always satisfies the Jacobi equation
\begin{equation}\label{eq:cotangent-jacobi}
    \ddot{y} + V(\lambda)\dot{y}+\mathbb{K}y=0.
\end{equation}
\end{lemma}

\begin{proof}
    By the definition of the inverse transpose $d_v\varphi_t^{-\top}$, we have
$$ \eta(t)\coloneqq d_v\varphi_t^{-\top}(\xi) \in T^*_{\varphi_t(v)}(SM)$$
if and only if 
\begin{equation*}
\xi = \eta(t)\circ d_v\varphi_t.
\end{equation*}
Therefore, we may write
$$\xi=x(t)\beta\circ  d_v\varphi_t+y(t)\phi_p\circ d_v\varphi_t=x(t)\varphi_t^*\beta+y(t)\varphi_t^*\phi_p.$$
Differentiating this identity with respect to $t$ and using the definition of the Lie derivative $\mathcal{L}_F$, we obtain
$$0=\dot{x}\varphi_t^*\beta+x\varphi_t^*(\mathcal{L}_F \beta)+\dot{y}\varphi_t^*\phi_p+y\varphi_t^*(\mathcal{L}_F \phi_p).$$
Using \eqref{eq:structure-equations-2}, we see that
$$0=(\dot{x}+px-\kappa_p y)\varphi_t^*\beta+(\dot{y}+x+(V\lambda-p)y)\varphi_t^*\phi_p.$$
Since $(\beta, \phi_p)$ are linearly independent, we get the pair of equations \eqref{eq:flow-on-sigma}, as desired. 
\end{proof}

\begin{comment}
{\color{red}
\begin{equation*}
    \begin{alignedat}{1}
    0 & = \ddot{y}+(F(V(\lambda)-\dot{p})y+(V(\lambda)-p)\dot{y}+\dot{x} \\
    &= \ddot{y}+(F(V(\lambda)-\dot{p})y+(V(\lambda)-p)\dot{y}+\kappa_p y -px\\
    &=\ddot{y}+(F(V(\lambda)-\dot{p})y+(V(\lambda)-p)\dot{y}+\kappa_p y +p\dot{y}+p(V(\lambda)-p)y\\
    &=\ddot{y}+V(\lambda)\dot{y}+(\kappa_p-\dot{p}-p(p-V(\lambda))+F(V(\lambda)))y\\
    \end{alignedat}
\end{equation*}}
\end{comment}

\begin{remark} \label{rem:correspondence}
    Observe that this implies that $x$ is completely determined by $y$.
    
    %Therefore, there is a one-to-one correspondence between the functions $y$ and 
    
    %$y(t;v)$ is given, the calculations in the proof of Lemma \ref{lem:coordinates} imply that $x_p(t;v)$ is completely determined. In particular, this gives us a one-to-one correspondence between functions $y_p(t;v)$ and points $\xi \in \Sigma(v)$.

%They capture all the information of the lifted dynamics in $\Sigma$.% of the cohorizontal and adapted covertical components. 
\end{remark}

%\begin{remark} \label{rem:correspondence}
  % Observe that this implies that $x$ is completely determined by $y$.
    
    %Therefore, there is a one-to-one correspondence between the functions $y$ and 
    
    %$y(t;v)$ is given, the calculations in the proof of Lemma \ref{lem:coordinates} imply that $x_p(t;v)$ is completely determined. In particular, this gives us a one-to-one correspondence between functions $y_p(t;v)$ and points $\xi \in \Sigma(v)$.

%They capture all the information of the lifted dynamics in $\Sigma$.% of the cohorizontal and adapted covertical components. 
%\end{remark}

For each $v\in SM$, it will also be useful to introduce the damping function 
\begin{equation}\label{eq:m-function}
    m(t)\coloneqq\exp\left(-\dfrac{1}{2}\int_0^t (V\lambda)(\varphi_\tau(v))\, d\tau\right).
\end{equation}
Indeed, it allows us to associate to each $\xi\in \Sigma(v)$ a new function $z\in \mathcal{C}^\infty(\R, \R)$ given by the relation
\begin{equation}\label{eq:dfn-damped-z}
z(t)\coloneqq\dfrac{y(t)}{m(t)}.
\end{equation}
We think of $z$ as a damped $y$ component. %, which we call the \textit{damped adapted covertical component}.

\begin{lemma}\label{lemma:normal-form}
    For each $\xi\in \Sigma(v)$, the $z$ component is a solution of the Jacobi equation
    \begin{equation}\label{eq:normalized-jacobi}
        \ddot{z}+
        \tilde{\kappa} z = 0
        %\left(\kappa+\dfrac{F(V(\lambda))}{2}-\dfrac{(V(\lambda))^2}{4}\right) z=0
    \end{equation}
    along the orbit of $v$, where the quantity $\tilde{\kappa}$ is defined in \eqref{eq:definition-of-damped-thermostat-curvature}.
\end{lemma}

\begin{proof}
%$$\dot{m}=-\dfrac{V(\lambda)}{2}m,$$
First, we use the fact that $y=mz$ to get
$$\dot{y}=-\dfrac{V\lambda}{2}y+m\dot{z}.$$
Taking a second derivative, we obtain
\begin{equation}\label{eq:derivative-of-y}
    \begin{alignedat}{1}
\ddot{y}%&=-\dfrac{F(V(\lambda))}{2}y-\dfrac{V(\lambda)}{2}\dot{y}+m\ddot{z}+\dot{m}\dot{z}\\
&=-\dfrac{FV\lambda}{2}y-\dfrac{V\lambda}{2}\left(-\dfrac{V\lambda}{2}y+m\dot{z}\right)-\dfrac{V\lambda}{2}m\dot{z}+m\ddot{z}\\
&=m\left(-\dfrac{FV\lambda}{2}z+\dfrac{(V\lambda)^2}{4}z-V(\lambda)\dot{z}+\ddot{z}\right).\\
    \end{alignedat}
\end{equation}
However, the Jacobi equation \eqref{eq:cotangent-jacobi} yields
\begin{equation*}
    \begin{alignedat}{1}
    \ddot{y}&=-V(\lambda)\dot{y}-(\kappa_0+FV\lambda)y\\
    &=m\left(\dfrac{(V\lambda)^2}{2}z-V(\lambda)\dot{z}-(\kappa_0+FV\lambda)z\right).\\
    \end{alignedat}
\end{equation*}
Setting them equal to each other then gives us the claim since $m$ is nowhere-vanishing.
\end{proof}

%\begin{remark}\label{rem:curvature-choice}
%Note that the notion of curvature appearing in \eqref{eq:normalized-jacobi} is simply $\kappa_p$ for the particular choice of $p=V(\lambda)/2$. 
%\end{remark}

One of the advantages of studying the damped component $z$ defined in \eqref{eq:dfn-damped-z} instead of $y$ is that, thanks to Lemma \ref{lemma:normal-form}, we are able to use the following general result when there are no conjugate points.

\begin{lemma}\label{lemma:generic-jacobi}
Let $k\in \mathcal{C}^\infty(\R)$ be such that any non-trivial solution $z$ to the equation
\begin{equation}\label{eq:general-jacobi-equation}
    \ddot{z}+kz=0
\end{equation}
vanishes at most once. If such $z$ vanishes once, then $|z(t)|$ is unbounded as $t\to\pm\infty$.
\end{lemma}

\begin{proof} By normalizing if needed, it suffices to consider the case where $z(0)=0$ and $\dot{z}(0)=1$. For each $t_0\neq 0$, we know thanks to the one-time vanishing property and the homogeneity of \eqref{eq:general-jacobi-equation} that there exists a unique solution $z_{t_0}$  with $z_{t_0}(0)=1$ and $z_{t_0}(t_0)=0.$

We claim that 
$$z_{t_0}=z_{-1}+\dfrac{z_{t_0}(-1)}{z(-1)}z.$$
Indeed, since both sides satisfy \eqref{eq:general-jacobi-equation} and agree at $t=-1$ and $t=0$, the one-time vanishing property tells us that they must agree for all $t\in \R$.

Differentiating with respect to $t$ and setting $t=0$, we obtain 
$$\dot{z}_{t_0}(0)=\dot{z}_{-1}(0)+\dfrac{z_{t_0}(-1)}{z(-1)}.$$
Now, let $t_0>0$. Since $z_{t_0}(-1)>0$ and $z(-1)<0$, we notice that $t_0\mapsto\dot{z}_{t_0}(0)$ is bounded above as $t_0\to \infty$. For $t>0$, note that the function
\begin{equation*}
    t\mapsto z(t)\int_t^{t_0} \dfrac{1}{(z(\tau))^2}\, d\tau
\end{equation*}
satisfies \eqref{eq:general-jacobi-equation}. Moreover, using a Taylor expansion, we can see that it tends to $1$ as $t\to 0$. Since it also agrees with the function $z_{t_0}$ at $t=t_0$, the one-time vanishing property yields
\begin{equation} \label{eq:z-int}
    z_{t_0}(t)=z(t)\int_t^{t_0} \dfrac{1}{(z(\tau))^2}\, d\tau
\end{equation}
for all $t\in \R$. It follows that, for $t_0>t_0'>0$, we have
\begin{equation}\label{eq:z-dot-diff}
    \dot{z}_{t_0}(0)-\dot{z}_{t_0'}(0)=\int_{t_0'}^{t_0}\dfrac{1}{(z(\tau))^2}\, d\tau >0,
\end{equation}
so $t_0\mapsto \dot{z}_{t_0}(0)$ is monotone increasing as $t_0\to \infty$. Combined with the previous upper bound, this implies that $t_0\mapsto\dot{z}_{t_0}(0)$ converges, so we may take the limit $t_0\to \infty$ in \eqref{eq:z-dot-diff} to obtain a convergent integral on the right-hand side. We conclude that $|z(t)|$ is unbounded as $t\to \infty$. The same argument works for $t\to -\infty$ if we instead take $t_0<0$.
\end{proof}

\subsection{Cocyles} \label{subsection:cocycles}
Using the global coframe $(\beta, \psi_\lambda)$ for the characteristic set $\Sigma$, we get an identification $\Sigma\cong M\times \R^2$. Therefore, for each $t\in \R$, we obtain a unique map $\Psi_t: SM\to \text{GL}(n,\R)$ characterized by
$$\tilde{\varphi}_t(v, \xi)=(\varphi_t(v), \Psi_t(v)\xi)$$
for all $(v, \xi)\in\Sigma\cong M\times \R^2$. The map $\Psi:SM\times \R\to \text{GL}(2,\R)$ satisfies the \emph{cocycle} property over the flow $\{\varphi_t\}_{t\in \R}$, that is,
$$\Psi_{t+s}(v)=\Psi_s(\varphi_t(v))\Psi_t(v)$$
for all $t, s\in \R$. This is simply a different point of view of the previous subsection, with the explicit relationship given by
$$\Psi_t(v)  \begin{pmatrix} x(0) \\ y(0) \end{pmatrix} = \begin{pmatrix} x(t) \\ y(t) \end{pmatrix},$$
where the $x$ and $y$ components satisfy the differential equations \eqref{eq:flow-on-sigma}. If we denote by $\Phi: SM\to \mathfrak{gl}(2, \R)$ the infinitesimal generator of $\{\Psi_t\}_{t\in \R}$, namely,
$$\Phi(v) \coloneqq \frac{d}{dt} \Big|_{t=0} \Psi_t(v),$$
then \eqref{eq:flow-on-sigma} with $p=0$ allows us to explicitly write
$$\Phi = \begin{pmatrix} 0 & \kappa_0 \\ -1 & -V\lambda \end{pmatrix}.$$

Of course, we could have made a different choice of global coframe on $\Sigma$. This is represented by a \emph{gauge}, that is, a smooth map $P:SM\to \text{GL}(2, \R)$, which gives rise to a new cocyle  over the flow $\{\varphi_t\}_{t\in \R}$ by conjugation:
$$\widetilde{\Psi}_t(v)=P^{-1}(\varphi_t(v))\Psi_t(v)P(v).$$
One can check that the new infinitesimal generator $\widetilde{\Phi}$ is related to $\Phi$ by
$$\widetilde{\Phi}=P^{-1}(\Phi  + F) P.$$
Under this lens, our previous choice of coframe $(\beta, \phi_p)$ corresponds to the gauge
$$P=\begin{pmatrix} 1 & 0 \\ -p & 1 \end{pmatrix},$$
and the infinitesimal generator of the new cocyle $\{\widetilde{\Psi}_t\}_{t\in \R}$ becomes
\begin{equation*}
\widetilde{\Phi} = \begin{pmatrix} -p & \kappa_p \\ -1 & p-V\lambda  \end{pmatrix}.
\end{equation*}
Note in particular that, for $p=V\lambda/2$, we may rewrite this as
    \begin{equation}\label{eq:split-of-cocycle}
        \widetilde{\Phi} = - \frac{V\lambda}{2} \text{Id} + \begin{pmatrix} 0 & \tilde{\kappa} \\ -1 & 0 \end{pmatrix}, 
    \end{equation}
    so that 
    \begin{equation}\label{eq:split-of-cocycle-exp}
        \widetilde{\Psi}_t(v) = m(t) \Gamma_t(v),
    \end{equation}
    where the damping function $m$ is defined in \eqref{eq:m-function} and $\{\Gamma_t\}_{t\in \R}$ is the cocycle generated by the last matrix in \eqref{eq:split-of-cocycle}. In light of Lemma \ref{lemma:normal-form}, we have 
    $$\Gamma_t(v)  \begin{pmatrix} x(0) \\ z(0) \end{pmatrix} = \begin{pmatrix} x(t) \\ z(t) \end{pmatrix}.$$
    Note that the infinitesimal generator of $\{\Gamma_t\}_{t\in \R}$ has trace zero, so $\Gamma_t: SM\to \text{SL}(2,\R)$. This is essentially the same cocyle as the one we would get from a geodesic flow with `Gaussian curvature' $\tilde{\kappa}$; however, note that $\tilde{\kappa}$ is a function on $SM$ instead of $M$. By picking the right gauge and damping the $y$ component in the cotangent bundle, we are hence reducing the problem to something that resembles the geodesic case. %By dampening the , we have thus effectively reduced the situation to studying the geodesic case 

%The Anosov and dominated splitting properties can be reformulated 

%One of the advantages of working with cocycles is that we can test when things are Anosov

\subsection{Conjugate points} \label{subsection:conjugate}

In the geodesic case, the definition of conjugate points is often formulated in terms of a Jacobi equation. Let us restate our previous definition of conjugate points in terms of a Jacobi equation for thermostats.

%Let $x_0, x_1 \in M$ be a pair of points such that there exists a thermostat geodesic $\gamma : [0,T] \rightarrow M$ with $\gamma(0) = x_0$ and $\gamma(T) = x_1$. Recall that $x_0$ and $x_1$ are conjugate along $\gamma$ if $\exp_{x_0}^\lambda$ is singular at $T \dot{\gamma}(0)$. With this in mind, we observe the following.

\begin{lemma} \label{lem:conjugate_along}
    Let $\gamma : [0,T] \rightarrow M$ be a thermostat geodesic segment with distinct endpoints $x_0=\gamma(0)$ and $x_1=\gamma(T)$. The points $x_0$ and $x_1$ are conjugate along $\gamma$ if and only if there exists a non-trivial solution $y$ to the Jacobi equation \eqref{eq:cotangent-jacobi} satisfying $y(0) = y(T) = 0$.
\end{lemma}

\begin{proof}
Since the function $m$ defined in \eqref{eq:m-function} is nowhere-vanishing, we have $y(t)=0$  if and only if $z(t)=0$. We can thus conclude by applying Lemma \ref{lemma:normal-form} and \cite[Theorem 4.3]{assylbekov14}. Indeed, while their Jacobi equation for $y$ is different from \eqref{eq:cotangent-jacobi} because they are working in the tangent bundle, the authors show in \cite[\S5]{assylbekov14} that a change of variables puts their Jacobi equation in the same normal form as \eqref{eq:normalized-jacobi}. 
\end{proof}

\begin{corollary}\label{corollary:conjugate-points-characterization}
    A thermostat has no conjugate points if and only if there are no non-trivial solutions to the Jacobi equation \eqref{eq:cotangent-jacobi} (or \eqref{eq:normalized-jacobi}) which vanish at two distinct points.
\end{corollary}

\begin{remark}
    There is a nice geometric interpretation of what is going on in the cotangent bundle. Note that having $y(t)=0$ is equivalent to $d_v\varphi_t^{-\top}(\xi)\in \mathbb{H}^*(\varphi_t(v))$, so we get part (a) of Theorem \ref{theorem:green-bundles-definition}. See Figure \ref{figure:three-cases}(a). 
\end{remark}

\begin{comment}
We can now highlight one of the advantages of working in the cotangent bundle instead of the tangent bundle. Since $\Sigma$ is invariant by the lifted dynamics, we know that $d\varphi_t^{-\top}(v)(\beta)\subset \Sigma$ for all $t\in \R$. This observation gives us our first key result.

\begin{proof}[Proof of Theorem \ref{theorem:dominated_splitting_without_conj}]
    Given the definitions of $(E^s)^*$ and $(E^u)^*$ in \eqref{eq:definition-cotangent-bundles}, it follows that $\Sigma=(E^s)^*\oplus (E^u)^*$. %Note that, 
    As a consequence of Lemma \ref{lemma:key-lemma}, 
    we have the relationship
    ${(E^{s/u})^*(v)\cap \mathbb{H}^*(v)=\{0\}}$ for all $v\in SM$. Given Corollary \ref{corollary:conjugate-points-characterization}, the generalized thermostat must have no conjugate points: indeed, the only way $d\varphi_t^{-\top}(\mathbb{H}^*(\varphi_{-t}(v))$ can make a full turn in $\Sigma$ is by intersecting $(E^{s/u})^*$ at some point, but these bundles are flow-invariant.
\end{proof}

We can also further exploit the fact that $y(t)=0$ if and only if $z(t)=0$.
\end{comment}

Armed with this perspective on conjugate points, we prove Theorem \ref{theorem:non-positive-implies-no-conjugate}. 

\begin{proof}[Proof of Theorem \ref{theorem:non-positive-implies-no-conjugate}]
    Let $z$ be a non-trivial solution to the Jacobi equation \eqref{eq:normalized-jacobi} and define $$w \coloneqq \left(\dot{z} - \left(p-\dfrac{V\lambda}{2}\right)z\right)z.$$ By Lemma \ref{lemma:normal-form}, we have
    \begin{equation*}
        \begin{alignedat}{1}
            \dot{w}%&=(-\tilde{\kappa}z-F(q)z-q\dot{z})z+(\dot{z}-qz)\dot{z}\\
            &=\dot{z}^2-2\left(p-\dfrac{V\lambda}{2}\right)\dot{z}z-\left(\tilde{\kappa}+F\left(p-\dfrac{V\lambda}{2}\right)\right)z^2\\
            &=\left(\dot{z}-\left(p-\dfrac{V\lambda}{2}\right)z\right)^2-\left(\tilde{\kappa}+F\left(p-\dfrac{V\lambda}{2}\right)+\left(p-\dfrac{V\lambda}{2}\right)^2\right)z^2\\
            &=\left(\dot{z}-\left(p-\dfrac{V\lambda}{2}\right)z\right)^2-\kappa_p z^2.
        \end{alignedat}
    \end{equation*}
    Since $\kappa_p \leq 0$ by assumption, we get $\dot{w}\geq 0,$
    %\[ \dot{q} = - \mathbb{K} z^2 + \left(\dot{z} + \frac{V(\lambda)}{2} z \right)^2 \geq 0,\]
    so the function $w$ is non-decreasing. Suppose for contradiction, using Corollary \ref{corollary:conjugate-points-characterization}, that $z$ vanishes multiple times. Note that, because of the Jacobi equation \eqref{eq:normalized-jacobi}, if $z$ vanishes on an interval, then we must have $z = 0$ everywhere; thus, we may assume that $z$ vanishes on a discrete set. Let $t_0 \in \R$ be such that $z(t_0) = 0$ and let $t_1 \coloneqq \inf\{t > t_0 \ | \ z(t) = 0\}$. If $t_1 = t_0$, then we have an infinite sequence $t_n \to t_0^+$ such that $z(t_n) = 0$. By the mean value theorem, we get a sequence $t_n' \to t_0^+$ with $\dot{z}(t_n') = 0$ and hence, $\dot{z}(t_0) = z(t_0) = 0$ by continuity, forcing $z$ to be zero everywhere. Thus, we must have $t_1 > t_0$.
    
    By construction, $z$ does not vanish on the interval $(t_0, t_1)$. Since $w$ is non-decreasing and $w(t_0)=w(t_1)=0$, $w$ must vanish on this interval. Then, however, $z$ solves the first order differential equation $\dot{z}=(p-V\lambda/2)z$ on $(t_0,t_1)$ with $z(t_0) = z(t_1) = 0$; it is easy to see that this implies that $z = 0$ on the interval, which is a contradiction.
\end{proof}

\subsection{Green bundles} \label{subsection:green_bundles} Next, we want to show that having no conjugate points implies that the subbundle $d\varphi_t^{-\top}( \mathbb{H}^*(\varphi_{-t}(v)))$ converges as $t \rightarrow \pm \infty$. In this paper, when we talk about convergence of subbundles, we mean it in the sense that the subbundles converge in the Grassmannian topology of the projective bundle $\mathbb{P}(\Sigma)$ (also called the \textit{Grassmann $1$-plane bundle}) of the vector bundle $\Sigma\to SM$. That is, $\mathbb{P}(\Sigma)$ is the $4$-dimensional manifold obtained by projectivizing $\Sigma(v)$ for each $v\in SM$: the fiber over $v$ consists of all $1$-dimensional subspaces of $\Sigma(v)$. Note that both $\mathbb{V}^*$ and $\mathbb{H}^*$ define sections of this bundle. The flow $\{\tilde{\varphi}_t\}_{t\in \R}$ on $T^*(SM)$ naturally induces a flow on $\mathbb{P}(\Sigma)$, which we continue to denote by the same symbol.

\begin{proof}[Proof of Theorem \ref{theorem:green-bundles-definition}]
    Fix $v\in SM$. 
    We have already shown part (a), so in what follows we may assume that $t\mapsto \pi (\varphi_t(v))$ contains no conjugate points on $M$. Equivalently, any solution $z$ to the Jacobi equation \eqref{eq:normalized-jacobi} vanishes at most once. 

    Let $z_{t_0}$ be as in the proof of Lemma \ref{lemma:generic-jacobi}. Using Remark \ref{rem:correspondence}, we see that $z_{t_0}$ uniquely determines a point $\xi_{t_0} \in \Sigma(v)$ so that $\R \xi_{t_0} = d\varphi_{t_0}^{-\top}( \mathbb{H}^*(\varphi_{-t_0}(v)))$. Note that the proof of Lemma \ref{lemma:generic-jacobi} shows that $\dot{z}_{t_0}(0)$ converges as $t_0 \rightarrow \pm \infty$. Using the continuous dependence of solutions to \eqref{eq:normalized-jacobi} on initial conditions, we have that the functions $z_{t_0}$ converges as $t_0 \rightarrow \pm \infty$ to solutions $z_{\pm \infty}$ of \eqref{eq:normalized-jacobi} whose corresponding points $\xi_{\pm \infty} \in \Sigma(v)$ must span $\lim_{t_0 \rightarrow \pm \infty} d\varphi_{t_0}^{-\top}( \mathbb{H}^*(\varphi_{-t_0}(v)))$ and the result follows. The transversality condition \eqref{eq:h-transversality} is then a direct consequence of Lemma \ref{lemma:key-lemma} and the no-conjugate-points assumption.
\end{proof}

Thus, if a thermostat has no conjugate points, then for each $v\in SM$, we can define $G^*_{s/u}(v)\subset \Sigma(v)$ by the limiting procedures \eqref{eq:green-bundles-definition}. Thanks to the transversality condition \eqref{eq:h-transversality}, we see that for each Borel measurable function $p: SM\to \R$ that is smooth in the direction of the flow, there exist functions $r^s, r^u : SM \rightarrow \R$ so that 
\begin{equation}\label{eq:definition-of-r}
    r^{s/u} \beta + \phi_p \in G_{s/u}^*.
\end{equation} In general, these functions are Borel measurable. Still, as the next lemma shows, they satisfy a Riccati equation in the flow direction, in which they are always smooth.  %For simplicity, we write $r_p$ for either $r_p^{s/u}$, and similarly we write $G^*$ for either $G^*_{s/u}$.

\begin{lemma}\label{lemma:riccati-equation}
Let $(M, g, \lambda)$ be a thermostat without conjugate points. For each Borel measurable function $p: SM\to \R$ smooth in the direction of the flow, the functions $r=r^{s/u}$ characterized by \eqref{eq:definition-of-r} satisfy the Riccati equation
\begin{equation}\label{eq:riccati-r-equation}
    r^2 + (V\lambda-2p)r + \kappa_p - Fr = 0.
\end{equation}

\end{lemma}
\begin{proof} Let us write $G^*$ for either $G^*_{s}$ or $G^*_u$. For notational convenience, let us also fix $v \in SM$ and define $\phi_p(t) \coloneqq \phi_p(\varphi_t(v))$, $r(t) \coloneqq r(\varphi_t(v))$, and $\beta(t) \coloneqq \beta(\varphi_t(v))$. If $\eta(t)\coloneqq d_v\varphi_{-t}^{-\top}\big(r(t)\beta(t) + \phi_p(t)\big),$
then %we have 
$\eta(t)\in G^*(v)$ for all $t\in \R$. Unraveling the definitions, this means that
$$d_v\varphi_{t}^{-\top}(\eta(t))=r(t)\beta(t)+\phi_p(t),$$
and therefore,
$$\eta(t)=r(t)\varphi_t^*\beta(t) +\varphi_t^*\phi_p(t).$$
Differentiating with respect to $t$ and setting $t=0$, we can use  \eqref{eq:structure-equations-2} to write
\begin{equation*}
    \begin{alignedat}{1}
\dot{\eta}(0) &=\dot{r}\beta+r\mathcal{L}_F\beta + \mathcal{L}_F\phi_p\\
& = (\dot{r} + r p - \kappa_p)\beta + (V\lambda - p + r)\phi_p.
    \end{alignedat}
\end{equation*}
Since $\eta(0)=r\beta + \phi_p$, we obtain
$$\dot{\eta}(0)-(V\lambda-p + r)\eta(0)=(\dot{r} - \kappa_p - r V\lambda+2rp- r^2)\beta.$$
The left-hand side belongs to $G^*$. We also know that $G^*$ is transverse to $\mathbb{H}^*$, so the right-hand side must be zero. The claim follows.
\end{proof}

We can now prove Theorem \ref{theorem:no-conjugate-points-characterization}.

\begin{proof}[Proof of Theorem \ref{theorem:no-conjugate-points-characterization}]
    One direction is given by Theorem \ref{theorem:non-positive-implies-no-conjugate}. In the other direction, suppose the thermostat has no conjugate points. Then, by picking $p=V\lambda$ in the Riccati equation \eqref{eq:riccati-r-equation}, Lemma \ref{lemma:riccati-equation} gives us a Borel measurable function $r$ smooth along the flow such that
$$r^2-rV\lambda + \kappa_0 + FV\lambda-Fr=0.$$
If we now define $p\coloneqq V\lambda-r$, we get
$$\kappa_p=\kappa_0 + Fp+p(p-V\lambda)=0,$$
as desired. 
\end{proof}

We also have the tools to prove Theorem \ref{theorem:hopf}.

\begin{proof}[Proof of Theorem \ref{theorem:hopf}]
Recall that $\text{div}_\mu F = V\lambda$. By Stokes's theorem, we have
    $$\int_{SM}F(r)\mu = -\int_{SM}rV(\lambda) \mu.$$
Integrating the Riccati equation \eqref{eq:riccati-r-equation}, we hence get
$$\int_{SM}(r^2+2(V\lambda-p)r)\mu =-\int_{SM}\kappa_p\mu.$$
It follows that
\begin{equation} \label{eqn:key}
    \begin{alignedat}{1}
\int_{SM}(\kappa_p - (p-V\lambda)^2) \mu &= -\int_{SM}(r^2+2(V\lambda-p)r+(p-V\lambda)^2) \mu \\
&=-\int_{SM}(r+V\lambda-p)^2\mu\leq 0.
    \end{alignedat}
\end{equation}
If the left-hand side is zero, then $r=p-V\lambda$ $\mu$-almost everywhere. Now, consider
\[ B_t \coloneqq \{v \in SM \ | \ r(\varphi_t(v)) =p(\varphi_t(v))- (V\lambda)(\varphi_t(v))\}, \qquad t\in \R.\]
Since the flow $\{\varphi_t\}_{t\in \R}$ is smooth, the Radon--Nikodym theorem implies that the subset $B_t$ has full measure for every $t \in \R$, and hence, the intersection $ B_0 \cap \big(\bigcap_{n \geq 1} B_{1/n}\big)$ also has full measure. Thus, we see that, for $\mu$-almost every $v \in SM$, we have
\begin{equation*}
\begin{alignedat}{1}
    (Fr)(v) &= \lim_{n \rightarrow \infty} n (r(\varphi_{1/n}(v)) - r(v)) \\
    &= \lim_{n \rightarrow \infty} n (p(\varphi_{1/n}(v))- (V\lambda)(\varphi_{1/n}(v)) - p(v) + (V\lambda)(v)) \\
    &= F(p-V\lambda)(v).
\end{alignedat}
\end{equation*}
%we see that $F(r) =- F(V(\lambda))$ almost everywhere. 
Substituting this into the Riccati equation \eqref{eq:riccati-r-equation} yields $\mathbb{K}=0$ $\mu$-almost everywhere. Since $\mathbb{K}$ is a smooth function, we get $\mathbb{K}=0$ everywhere.
\end{proof}

In fact, a slight modification of this argument  yields the following, which will be useful in the proof of Theorem \ref{theorem:collapse}.

\begin{lemma} \label{lem:gen-hopf}
    Let $(M,g,\lambda)$ be a thermostat without conjugate points. For any finite flow-invariant Borel measure $\nu$ on $SM$, we have 
    \[ \int_{SM} 
    %\left(\kappa + \frac{F(V(\lambda))}{2} - \frac{(V(\lambda))^2}{4}\right) 
    \tilde{\kappa}
    \, d\nu    
    \leq 0, \]
    with equality if and only if 
    %$\kappa + F(V(\lambda))/2 - (V(\lambda))^2/4=0$ 
    $\tilde{\kappa} = 0$
    $\nu$-almost everywhere.
\end{lemma}

\begin{proof}
    Taking $p = V\lambda/2$ and integrating the Riccati equation \eqref{eq:riccati-r-equation} with respect to the measure $\nu$, we get 
    \[ \int_{SM} 
    %\left(\kappa + \frac{F(V(\lambda))}{2} - \frac{(V(\lambda))^2}{4}\right)
    \tilde{\kappa}\,
    d\nu = - \int_{SM} r^2\, d\nu \leq 0.\]
    The left-hand side is zero if and only if $r$ is zero $\nu$-almost everywhere. With the same argument as in the proof of Theorem \ref{theorem:hopf}, we notice that having $r=0$ $\nu$-almost everywhere implies that $Fr=0$ $\nu$-almost everywhere. It then follows that $\tilde{\kappa}=0$ $\nu$-almost everywhere. 
\end{proof}

\subsection{Lyapunov exponents} \label{subsec:lyapunov}
Recall that the \emph{Lyapunov exponent} at $(v, \xi) \in \Sigma$ is
\begin{equation} \label{eq:definition-of-lyapunov} \chi(v, \xi) \coloneqq \limsup_{t \rightarrow \infty} \frac{1}{t} \ln \|d_v \varphi_t^{-\top}( \xi)\| ,\end{equation}
where $\Vert\cdot\Vert$ is any continuous metric norm on $\Sigma$. We use $\|d_v\varphi_t^{-\top}(\xi)\| \coloneqq |x(t)| + |y(t)|$ where  $x$ and $y$ are the adapted coordinates given by \eqref{eq:definition-of-x-y}. Let $u \coloneqq \dot{y}/y$ and $w \coloneqq \dot{z}/z$. Using the Jacobi equation \eqref{eq:normalized-jacobi} we see that, wherever $w$ is defined, it satisfies the Riccati equation
\begin{equation} \label{eq:damped-riccati} \dot{w} + w^2 + \tilde{\kappa}
%\left( \kappa + \frac{F(V(\lambda))}{2} - \frac{(V(\lambda))^2}{4} \right) 
= 0.\end{equation}
Furthermore, using \eqref{eq:derivative-of-y}, we get the relationship
\begin{equation} \label{eq:riccati-reln} u(t) = w(t) - \frac{1}{2}V(\lambda)(\varphi_t(v)).\end{equation}

We want to relate the exponential growth rate of $\|d_v\varphi_t^{-\top}(\xi)\|$ to the exponential growth rate of $|y(t)|$ in the case where $\xi \in G_{s/u}^*(v)$. For completeness, we recall the following standard Riccati comparison result (see, for example, \cite[Lemma 2.1]{green54} and the discussion that follows).

\begin{lemma} \label{lem:standard-ricc}
    Fix $\kappa \in \mathcal{C}^\infty(\R, \R)$ and let $w \in \mathcal{C}^1(\R, \R)$ solve the Riccati equation
    \[ \dot{w} + w^2 + \kappa = 0.\]
    If there is a constant $K > 0$ so that $\kappa \geq -K^2$, then $|w(t)| \leq K$ for all $t\in \R$.
\end{lemma}

In particular, compactness of $SM$ and the above lemma imply that globally defined solutions to \eqref{eq:damped-riccati} are bounded. This gives us the ingredients to prove the following lemma.

\begin{lemma} \label{lem:x-y-bound}
Let $(M, g, \lambda)$ be a thermostat without conjugate points. If $\xi \in G_{s/u}^*(v)$, then there is a constant $C > 0$ such that $|x(t)| \leq C |y(t)|$ for all $t\in \R$.
\end{lemma}

\begin{proof}
Since $\xi \in G_{s/u}^*(v)$, we know thanks to the transversality condition in Theorem \ref{theorem:green-bundles-definition} that $y$ never vanishes. Equations \eqref{eq:flow-on-sigma} then imply that 
    \[ \frac{x(t)}{y(t)} = (p-V\lambda)(\varphi_t(v)) - u(t). \]
%    {\color{red}By compactness, we have that there is a constant $\kappa < 0$ so that $\kappa \leq \tilde{\kappa}$. Apply Lemma \ref{lem:standard-ricc} with \eqref{eq:damped-riccati} to deduce that $w$ is bounded above.}
    %
    %{\color{red}A standard Riccati comparison result yields that that $w$ is bounded (e.g., \cite[Lemma 2.8]{eberlein73})}, and 
    Since $w$ is bounded, \eqref{eq:riccati-reln}, along with the fact that $V\lambda$ is bounded implies that $u$ is bounded. The claim follows since $p-V\lambda$ is also bounded along any orbit.
\end{proof}

The following dynamical criterion for the Green bundles will also be useful.

%We now observe a dynamical criterion for the Green bundles that will be useful throughout.

\begin{lemma}\label{lemma:bounded_implication}
    Let $(M,g,\lambda)$ be a thermostat without conjugate points and $\xi \in \Sigma(v)$. If the function $z$ is bounded for all $t\geq 0$ (respectively $t\leq 0$), then $\xi \in G^*_s(v)$ (respectively $\xi\in G^*_u(v)$). 
\end{lemma}

\begin{proof} Let $\xi \in \Sigma(v)$ and suppose $z$ is bounded for all $t\geq 0$. By normalizing, we may assume that $z(0) = 1$. Let $z_{t_0}$ and $w$ be the solutions to the Jacobi equation \eqref{eq:normalized-jacobi} satisfying $z_{t_0}(0) = 1$, $z_{t_0}(t_0) = 0$, $w(0) = 0$ and $\dot{w}(0) = 1$. There must be a family of constants $c_{t_0}\in \R$ such that $z = z_{t_0} + c_{t_0} w$, so it suffices to show that $c_{t_0} \rightarrow 0$ as $t_0 \rightarrow \infty$. To that end, observe that \eqref{eq:z-int} implies that, for any $t>0$, we have
    \[ \lim_{t_0 \rightarrow \infty} c_{t_0} = \lim_{t_0 \rightarrow \infty} \frac{z(t)-z_{t_0}(t)}{w(t)}  = \frac{z(t)}{w(t)} - \int_t^\infty \frac{1}{(w(\tau))^2}\, d\tau.\]
    Since $w$ is unbounded by Lemma \ref{lemma:generic-jacobi} and $z$ is bounded for all $t\geq 0$ by assumption, we get the desired conclusion by taking $t\to \infty$. The same argument with $t_0 \rightarrow -\infty$ gives us the claim if $z$ is bounded for all $t\leq 0$.
\end{proof}

For $\xi \in G_{s/u}^*(v)$, we can use Lemma \ref{lem:x-y-bound} and \eqref{eq:riccati-reln} to rewrite \eqref{eq:definition-of-lyapunov} as
\begin{equation} \label{eq:definition-of-lyapunov-2}
\begin{split}
    \chi(v,\xi) &= \limsup_{t \rightarrow \infty} \frac{1}{t} \ln|y(t)| \\
    &= \limsup_{t \rightarrow \infty} \frac{1}{t} \int_0^t u(\tau) \, d\tau \\
    &= \limsup_{t \rightarrow \infty} \frac{1}{t} \int_0^t \left(w(\tau) - \frac{1}{2} V(\lambda)(\varphi_\tau(v)) \right)\, d\tau.
    \end{split}
\end{equation}
Furthermore, for any $\xi \in G_{s/u}^*(v) \setminus \{0\}$, it is clear that $\chi(v, \xi)=\chi(v, \pm \xi/\Vert \xi\Vert)$. We write $\chi^{s/u}(v) \coloneqq \chi(v,\xi)$ for $\xi \in G_{s/u}^*(v) \setminus \{0\}$.

Let $\nu$ be a Borel measure on $SM$ which is ergodic for the flow. The Oseledets theorem \cite{osedelets68} says that the limit \eqref{eq:definition-of-lyapunov} exists $\nu$-almost everywhere. Furthermore, we have a splitting $\Sigma(v) = E_0(v)\oplus E_-(v) \oplus E_+(v)$ for $\nu$-almost every $v \in SM$, where
\[E_0(v) \coloneqq \{\xi \in \Sigma(v) \ | \ \chi(v,\xi) = 0\}  \qquad  \text{ and } \qquad E_{\pm}(v) \coloneqq \{\xi \in \Sigma(v) \ | \ \mp \chi(v,\xi) < 0\}. \]
By Lemma \ref{lemma:bounded_implication}, we have the inclusions
\begin{equation} \label{eqn:lyapunov-inclusions} 
E_+ \subseteq G_u^* \subset E_0 \oplus E_+ \qquad \text{ and }\qquad  E_- \subseteq G_s^* \subset E_0 \oplus E_-.
\end{equation}
Note that since we are on a surface and $G_{s/u}^*(v)$ are $1$-dimensional subspaces, we either have $E_+(v) = G_u^*(v)$ or $E_+(v) = \{0\}$. This leads us to the following result.

\begin{lemma} \label{lemma:lyapunov-collapse}
Let $(M, g, \lambda)$ be a thermostat without conjugate points and let $\nu$ be a Borel ergodic measure on $SM$. We have $G_s^*(v) = G_u^*(v)$ $\nu$-almost everywhere if and only if $\chi^u(v) = \chi^s(v) = 0$ $\nu$-almost everywhere. Furthermore, if $G_s^*(v) = G_u^*(v)$ for all $v \in SM$, then the topological entropy of the thermostat flow is zero.
\end{lemma}

\begin{proof}
The first statement follows immediately from the inclusions \eqref{eqn:lyapunov-inclusions}. Assume now that we have $G_s^*(v) = G_u^*(v)$ for every $v \in SM$. Using Ruelle's inequality \cite{ruelle78}, the metric entropy of the flow is zero with respect to any ergodic Borel measure and, hence, the metric entropy of any Borel invariant measure is zero using an ergodic decomposition argument (see, for example, \cite[Theorem 3.3.37]{fisher19}). The variational principle \cite[Theorem 4.3.7]{fisher19} implies that the topological entropy is zero.
\end{proof}

%We observe how collapsing of the Green bundles implies that the Lyapunov exponents are zero everywhere. 

%We can also explain what happens to the Green bundles when the thermostat curvature vanishes everywhere. 

%We can also show what happens to the Green bundles when the thermostat curvature

%We can also show a rigidity result which characterizes flat generalized thermostats as those without conjugate points whose Green bundles collapse to a line everywhere.

%We now show what happens to the Green bundles when the thermostat curvature vanishes everywhere, 

We now show that the Green bundles collapse to a line $\mu$-almost everywhere when the thermostat curvature vanishes everywhere. In the magnetic case, this collapsing happens everywhere and the converse also holds provided the curvature is non-positive.

\begin{proof}[Proof of Theorem \ref{theorem:collapse}]
    Suppose $\mathbb{K}=0$. By Theorem \ref{theorem:non-positive-implies-no-conjugate}, we know that the thermostat has no conjugate points. Using \eqref{eqn:key} with $p=V\lambda$, we get $r^{s/u}=0$ $\mu$-almost everywhere. It follows that $G^*_s=G^*_u=\R(\psi_\lambda-V(\lambda)\beta)$ $\mu$-almost everywhere.

%Recall that $d\varphi_t^{-\top}(v)\beta = x(t)\beta + y(t)\varphi_\lambda$
%{\color{red} Conclude}
%Whenever $x(0)=0$, since $m(t)=1$, we get $|z(t)|=|y(0)|$ for all $t\in \R$. Therefore, thanks to Lemma \ref{lemma:bounded_implication}, we know that $G^*_s=G^*_u=\R\phi_p=\R\psi_\lambda$.

%Therefore, given our dynamical characterization of the Green bundles in Lemma \ref{lemma:bounded_implication}, we conclude that $G^*_s(v)=G^*_u(v)=\{x\beta + y\phi_p \in T^*_v(SM)\mid x=0\}$. Since we are using the basis $\{\beta, \phi_p\}$ with $p=V(\lambda)$, this gives us $G^*_s=G^*_u=\R(\psi_\lambda-V(\lambda)\beta)$. 

%In the other direction, suppose that $(M, g, \lambda)$ has no conjugate points and $V(\lambda)=0$. Note that $\kappa = \mathbb{K}$ in this setting. 
If we pick $p=V\lambda$, then $\kappa_p=\mathbb{K}$ and \eqref{eq:flow-on-sigma} become 
\begin{equation*}
    \begin{cases}
\dot{x}+V(\lambda)x=0,\\
\dot{y}+x = 0.
    \end{cases}
\end{equation*}
We get the explicit solutions
\begin{equation}\label{eq:explicit-solutions}
    x(t)=\exp\left(-\int_0^tV(\lambda)(\varphi_\tau(v))\, d\tau\right)x(0), \qquad y(t)=y(0)-\int_0^tx(\tau)\, d\tau.
\end{equation}
In particular, when $V\lambda=0$, the solutions are $x(t)=x(0)$ and $y(t)=y(0)-t x(0)$. Therefore, if we start with $x(0)\neq 0$ and $y(0)=0$, which corresponds to a covector in $\mathbb{H}^*$,  we get
$$\lim_{t\to \pm \infty}\dfrac{x(t)}{y(t)}=\lim_{t\pm \infty}-\dfrac{1}{t}=0.$$
This tells us that the subbundle $d\varphi_{t}^{-\top}(\mathbb{H}^*(\varphi_{-t}(v)))$ converges to $\R\psi_\lambda$ as $t\to\pm \infty$. This holds everywhere, so we obtain $G^*_s=G^*_u=\R\psi_\lambda$, proving claim (a).

Next, observe that it suffices to prove claim (b) in the setting where $\nu$ is ergodic by taking an ergodic decomposition of the measure. Let $\nu$ be a Borel ergodic measure and let $\Delta \coloneqq \{v \in SM  \mid  \chi^u(v) > 0\}$. Whenever the limit exists, let
\[ \rho(v) \coloneqq \lim_{t \rightarrow \infty} \frac{1}{t} \int_0^t \tilde{\kappa}(\varphi_\tau(v))\,  d\tau.\]
By the Birkhoff ergodic theorem, this function is Borel measurable, constant $\nu$-almost everywhere, and 
\[ \int_{SM} \rho\,  d\nu  = \int_{SM} \tilde{\kappa}\, d\nu.\]
Thus, if we set $P^- \coloneqq \left\{ v \in SM \ | \ \rho(v) < 0\right\}$, then ergodicity implies that $\nu(P^-) = 1$ or $\nu(P^-) = 0$. Without loss of generality, we may assume that the limit \eqref{eq:definition-of-lyapunov-2} exists for every $v \in P^-$. Once we show that $P^- = \Delta$ $\nu$-almost everywhere, claim
 (b) will follow from Lemma \ref{lemma:lyapunov-collapse}. Indeed, if $\tilde{\kappa}(v) < 0$ for some $v \in SM$, then Lemma \ref{lem:gen-hopf} along with ergodicity implies that $\nu(P^-) = \nu(\Delta) = 1$ and, hence, $G_s^* \neq G_u^*$ $\nu$-almost everywhere. However, if $\tilde{\kappa} = 0$ $\nu$-almost everywhere, then we have $\nu(P^-) = \nu(\Delta) =0$ and, hence, $G_s^* = G_u^*$ $\nu$-almost everywhere. 

Integrating \eqref{eq:damped-riccati} from $\tau=0$ to $\tau=t$, we get
\[ \int_0^t (w(\tau))^2\,  d\tau + \int_0^t \tilde{\kappa}(\varphi_\tau(v)) \, d\tau + w(t) - w(0) = 0.\]
In particular, if $v \in P^-$ and $w$ corresponds to a covector $\xi \in G_u^*(v)$, then normalizing the above by $t$, taking the limit, and noting that $w$ is bounded for $t \geq 0$, we have
\[  \lim_{t \rightarrow \infty} \frac{1}{t} \int_0^t (w(\tau))^2\, d\tau = - \rho(v) > 0.\]
Since $\tilde{\kappa} \leq 0$, a comparison argument implies that $w \geq 0$ and a standard analysis lemma implies that $\chi^u(v) > 0$ (see \cite[Lemma 3.4.4]{barreira02}). However, if $v \in \Delta$, then the Cauchy--Schwarz inequality yields 
$$ \left(-\limsup_{t \rightarrow \infty} \frac{1}{t} \int_0^t \tilde{\kappa}(\varphi_\tau(v))\, d\tau\right)^{1/2} \geq \chi^u(v) > 0.$$
Thus, $\rho(v) > 0$ where the limit exists and, hence, $P^- = \Delta$ $\nu$-almost everywhere.\qedhere

\end{proof}

\section{Transverse Green bundles} \label{section:transversal_green_bundles}

The goal of this section is to prove Theorems \ref{theorem:characterization-dominated-splitting} and \ref{theorem:partial-converse}. The key property we need to show is that the Green bundles are continuous whenever they are transverse everywhere. For this, we adapt some of Eberlein's arguments in \cite{eberlein73} to this more general setting, giving a dynamical characterization of the stable and unstable Green bundles. We start with the following corollary of Lemma \ref{lemma:bounded_implication}.

\begin{corollary}\label{corollary:transversal-implication-1}
    Let $(M,g,\lambda)$ be a thermostat without conjugate points. Then, if $G_s^*(v) \cap G_u^*(v) = \{0\}$ for all $v \in SM$, there are no non-trivial bounded solutions to the Jacobi equation \eqref{eq:normalized-jacobi} for any $(v,\xi) \in \Sigma$.
\end{corollary}

The next step is to analyze the growth of solutions to \eqref{eq:normalized-jacobi} when they vanish, that is, when they correspond to vectors in the cohorizontal subbundle $\mathbb{H}^*$. 
\begin{lemma}\label{lemma:eberlein-3.7}
    Let $(M, g, \lambda)$ be a thermostat without conjugate points. Then, if we have $G_s^*(v) \cap G_u^*(v) = \{0\}$ for all $v \in SM$, there exists a constant $A>0$ such that if $z$ is a solution to \eqref{eq:normalized-jacobi} with $z(0)=0$, then $|z(t)|\geq A|z(\tau)|$ for all $t\geq \tau \geq 1$.
\end{lemma}
\begin{proof}
    We argue by contradiction. For each integer $n\geq 1$, pick a non-trivial solution $z_n$ to \eqref{eq:normalized-jacobi} and $t_n\geq \tau_n\geq 1$ such that $z_n(0)=0$ and $|z_n(t_n)|\leq (1/n)|z_n(\tau_n)|$. Multiplying by a constant if necessary, we can assume that $\dot{z}_n(0)=1$.
    
    For each $n\geq 1$, pick $u_n\in [0, t_n]$ such that $z_n(t)\leq z_n(u_n)$ for all $t\in [0,t_n]$. Let  $\delta\coloneqq\inf_{n\geq 1} u_n$. We must have $\delta>0$. If not, then $u_n\to 0$ up to picking a subsequence. However, then $\lim_{n\to \infty}z_n(u_n)=0$ by compactness of $SM$ and continuity. Nevertheless, since $t_n \geq 1$ for all $n$, we have $z_n(u_n)\geq z_n(1)$.  Given that $\inf_{n\geq 1} z_n(1)>0$ by compactness, this is a contradiction.

    Define $w_n(t)\coloneqq z_n(t+u_n)/z_n(u_n)$. Note that each function $w_n$ satisfies the Jacobi equation \eqref{eq:normalized-jacobi} and
\begin{equation*}
\begin{cases}
w_n(-u_n)=0,\\
w_n(0)=1,\\
w_n(t_n-u_n)\leq 1/n,\\
w_n(t)\leq 1\text{ for } -u_n\leq t\leq  t_n-u_n.
\end{cases}
\end{equation*}    
Using Lemma \ref{lem:standard-ricc}, we get that there is a uniform constant $K > 0$ so that $|\dot{w}_n(0)| \leq K$.
   Moreover, using compactness to work on a convergent subsequence, we have that $\lim_{n \rightarrow \infty} w_n = w$. This function $w$ also satisfies the Jacobi equation \eqref{eq:normalized-jacobi} and $w(0)=1$ by continuity, so it is a non-trivial solution. We now argue that we have a contradiction in every possible case.
    \begin{enumerate}
    \item[(1)] $\{t_n-u_n\}_{n\in \N}$ and $\{u_n\}_{n\in \N}$ both contain bounded subsequences. Up to picking subsequences, we have $t_n-u_n\to t$ and $u_n\to u$. By continuity, we obtain $w(-u)=0=w(t)=0$, but this is impossible since $-u\leq -\delta <0\leq t$. 
    
    \item[(2)] $u_n\to \infty$ and $\{t_n-u_n\}_{n\in \N}$ contains a bounded subsequence. Up to taking a subsequence, we have $t_n-u_n\to t$. By continuity, $w(t)=0$ and $w(\tau)\leq 1$ for all $\tau\leq t$. This contradicts  Lemma \ref{lemma:generic-jacobi}.
    
    \item[(3)]$t_n-u_n\to \infty$ and $\{u_n\}_{n\in \N}$ contains a bounded subsequence. We can use the same argument as in part (2).
    
    \item[(4)]$t_n-u_n\to \infty$ and $u_n\to \infty$. This implies that $|w(t)|\leq 1$ for all $t\in \R$ by continuity. This contradicts Corollary \ref{corollary:transversal-implication-1}. \qedhere
    \end{enumerate}
\end{proof}

   % It suffices to show that, for any $\xi \in \Sigma(v)$, if the corresponding solution to \eqref{eq:normalized-jacobi} is bounded for all $t \geq 0$, then $\xi \in G_s^*(v)$. The same argument will show that if the corresponding solution to \eqref{eq:normalized-jacobi} is bounded for all $t \leq 0$, then $\xi \in G_u^*(v)$. Thus any non-trivial bounded solution to \eqref{eq:normalized-jacobi} for some $v \in SM$ corresponds to $\xi \in G_s^*(v) \cap G_u^*(v)$. 

We have mentioned above that not all thermostats are reversible. However, reversing the orbit of a thermostat does give rise to the orbit of a thermostat where the intensity is `flipped', that is, its mirrored thermostat.   More precisely, let $\mathcal{F} : TM \rightarrow TM$  be the  \emph{flip map} given by $(x,v) \mapsto (x,-v)$ and for $\lambda \in \mathcal{C}^\infty(SM)$, define $\lambda^\mathcal{F} \coloneqq - \lambda \circ \mathcal{F}$. Unraveling the definition, it is clear that if $\gamma$ is a thermostat geodesic for $(M,g,\lambda)$, then the curve $t \mapsto \gamma(-t)$ is a thermostat geodesic for $(M,g,  \lambda^\mathcal{F})$, with initial conditions $(\gamma(0), - \dot{\gamma}(0))$. Furthermore, reversing time and the initial velocity gives us a correspondence of solutions to the Jacobi equation \eqref{eq:normalized-jacobi} between the two thermostats. We summarize this observation with the following lemma.

\begin{lemma} \label{lem:pseudo_reversibility}
    Let $(M,g,\lambda)$ be a thermostat and let $\gamma$ be a thermostat geodesic. The curve $t \mapsto \gamma(-t)$ is a thermostat geodesic for $(M,g,  \lambda^\mathcal{F})$. Moreover, $z(t)$ is a solution to the Jacobi equation \eqref{eq:normalized-jacobi} along $t \mapsto \gamma(t)$ for $(M,g,\lambda)$ if and only if $z(-t)$ is also a solution to \eqref{eq:normalized-jacobi} along $t \mapsto \gamma(-t)$ for $(M,g,\lambda^\mathcal{F})$.
    %$(M,g,\lambda)$ is without conjugate points if and only if $(M,g, \lambda^\mathcal{F})$ is without conjugate points, and there is a non-trivial bounded solution to \eqref{eq:normalized-jacobi} for $(M,g,\lambda)$ if and only if there is a non-trivial bounded solution to \eqref{eq:normalized-jacobi} for $(M,g,\lambda^\mathcal{F})$.
\end{lemma}

\begin{comment}
\begin{lemma} \label{lem:pseudo_reversibility}
    Let $(M,g,\lambda)$ be a generalized thermostat, and let $\gamma$ be an orbit. The curve $t \mapsto \gamma(-t)$ is a generalized thermostat geodesic for the system $(M,g,  \lambda^\mathcal{F})$. Moreover, $(M,g,\lambda)$ is without conjugate points if and only if $(M,g, \lambda^\mathcal{F})$ is without conjugate points, and there is a non-trivial bounded solution to \eqref{eq:normalized-jacobi} for $(M,g,\lambda)$ if and only if there is a non-trivial bounded solution to \eqref{eq:normalized-jacobi} for $(M,g,\lambda^\mathcal{F})$.
\end{lemma}
\end{comment}

With this, we now have the ingredients to give a dynamical characterization of the Green bundles.

\begin{lemma} \label{lem:dynamical_criteria}
    Let $(M,g,\lambda)$ be a thermostat without conjugate points and suppose $G_s^*(v) \cap G_u^*(v) = \{0\}$ for all $v \in SM$. We have $\xi \in G_s^*(v)$ (respectively $\xi\in G_u^*(v)$) if and only if the corresponding solution $z$ to the Jacobi equation \eqref{eq:normalized-jacobi} is bounded for all $t \geq 0$ (respectively $t \leq 0$). 
\end{lemma}

\begin{proof}
Let $\xi \in G_s^*(v)$ correspond to the solution $z$ to \eqref{eq:normalized-jacobi} with $z(0) = 1$, and define $z_{t_0}$ as the solution to \eqref{eq:normalized-jacobi} with $z_{t_0}(0) = 1$ and $z_{t_0}(t_0) = 0$. If we set $u_{t_0}(t) \coloneqq z_{t_0}(-t)$, then $u_{t_0}$ is a solution to the Jacobi equation \eqref{eq:normalized-jacobi} along the orbit $t \mapsto \gamma(-t)$ for the thermostat $(M,g,\lambda^\mathcal{F})$. In particular, fixing $t \geq 0$ and considering $t_0 \geq 1 + t$, Lemmas \ref{lemma:eberlein-3.7} and \ref{lem:pseudo_reversibility} imply
\[ |z_{t_0}(t)| = |u_{t_0}(-t)| = |u_{t_0}(-t_0 + (t_0-t))| \leq \frac{1}{A}.\]
By taking $t_0 \rightarrow \infty$, we get $|z(t)| \leq 1/A$ for all $t \geq 0$. The converse is the content of Lemma  \ref{lemma:bounded_implication}.
\end{proof}

\begin{remark} \label{rem:upper_bound}
    Notice that the argument above shows that if $\xi \in G_s^*(v)$, then the corresponding solution $z$ to \eqref{eq:normalized-jacobi} satisfies $|z(t)| \leq z(0)/A$ for all $t \geq 0$, where $A$ is the uniform constant from Lemma \ref{lemma:eberlein-3.7}. 
\end{remark}

In general, one knows very little about the regularity of the Green bundles. We can now prove that if they are transverse everywhere, they must at least be continuous.

\begin{proposition} \label{prop:continuity_green}
    Let $(M,g,\lambda)$ be a thermostat without conjugate points. Then, if $G_s^*(v) \cap G_u^*(v) = \{0\}$ for all $v \in SM$, the Green bundles are continuous. 
\end{proposition}

\begin{proof} We introduce the notation $z(t ;v)$ to emphasize that we are dealing with a solution to the Jacobi equation \eqref{eq:normalized-jacobi} corresponding to some $\xi \in \Sigma(v)$, and similarly for $m(t ;v)$ and $y(t ;v)$.

Let $\{v_n\}_{n\in \N} \subset SM$ be a sequence such that $v_n \rightarrow v$. It suffices to show that $G_s^*(v_n) \rightarrow G_s^*(v)$. To that end, let $z_n(t;v_n)$ be a sequence of solutions to  \eqref{eq:normalized-jacobi} such that the corresponding covectors $\xi_n$ lie in $G_s^*(v_n)$. By normalizing if needed, we may use compactness to assume that $\xi_n \rightarrow \xi \in \Sigma(v)$. Further normalize so that $z_n(0;v_n) = 1$ for each $n\in \N$. In coordinates, this implies that $y_n(t;v_n) \rightarrow y(t;v)$ and $\dot{y}_n(t;v_n) \rightarrow \dot{y}(t;v)$ for every $t \in \R$. Let $z(t;v)$ be a solution to \eqref{eq:normalized-jacobi} corresponding to $\xi$. Suppose for the sake of contradiction that $z(t;v)$ is not bounded above for all $t \geq 0$. Observe that 
\[ \begin{split} |z(t;v) - z_n(t;v_n)| & \leq \frac{|m(t;v_n) y(t;v) - m(t;v) y_n(t;v_n)|}{m(t;v) m(t;v_n)} \\
& \leq \frac{|y(t;v) - y_n(t;v_n)|}{m(t;v_n)} + \frac{|m(t;v_n) - m(t;v)| |y(t;v)|}{m(t;v) m(t;v_n)},\end{split}\]
and for fixed $t$, the right-hand side tends to zero as $n\to\infty$. By choosing sufficiently large $t$, we have $z_n(t;v_n) \geq 1/A$ for sufficiently large $n$, contradicting Remark \ref{rem:upper_bound}.
\end{proof}

Finally, we prove Theorem \ref{theorem:partial-converse}. 

\begin{proof}[Proof of Theorem \ref{theorem:partial-converse}]
Let $p=0$ and define $w\coloneqq-(r^s+r^u)/2$, where $r^{s/u}$ are the functions defined by \eqref{eq:definition-of-r}. Since both functions $r^s$ and $r^u$ satisfy the Riccati equation \eqref{eq:riccati-r-equation}, we get
\begin{equation*}
\begin{alignedat}{1}
\kappa_w &= w^2-V(\lambda)w + \kappa_0 + Fw \\
&=\dfrac{\left(r^s\right)^2}{4}+\dfrac{\left(r^u\right)^2}{4}+\dfrac{r^s r^u}{2}-V(\lambda)w + \kappa_0 + Fw	\\
%&=\dfrac{\left(r^s\right)^2}{2}-\dfrac{\left(r^s\right)^2}{4}+\dfrac{\left(r^u\right)^2}{2}-\dfrac{\left(r^u\right)^2}{4}+\dfrac{r^s r^u}{2}+(V(\lambda)-2p)w + \kappa_p -F(w)\\
&=-\dfrac{(r^s)^2}{4}-\dfrac{(r^u)^2}{4}+\dfrac{r^s r^u }{2}\\
%&=-\dfrac{1}{4}((r^2)^2 - 2r^sr^u + (r^u)^2)\\
&=-\dfrac{1}{4}(r^s - r^u)^2.
\end{alignedat}
\end{equation*}
Since $r^s(v) \neq r^u(v)$ for all $v\in SM$ by the transversality condition, it follows that $\kappa_{w}< 0$. In terms of regularity, the functions $r^{s/u}$ are always smooth along the flow. However, thanks to Proposition \ref{prop:continuity_green}, we also know that they are continuous since the Green bundles are transverse everywhere by assumption.  
\end{proof}

This allows us to then prove our main result.

\begin{proof}[Proof of Theorem 
\ref{theorem:characterization-dominated-splitting}]
 Given the definitions of $E_s^*$ and $E_u^*$ in \eqref{eq:definition-cotangent-bundles}, we previously noted that $\Sigma=E_s^*\oplus E_u^*$. The estimates \eqref{eq:dominated-estimates} then tell us that $G^*_s=E_s^*$ and $G_u^*=E_u^*$. By Theorem \ref{theorem:partial-converse}, we know that there exists a continuous function $p: SM\to \R$, smooth along the flow, such that $\kappa_p <0$. We then apply \cite[Theorem 3.7]{mettler19}. 
\end{proof}

\section{Examples and counterexamples on the $2$-torus}\label{section:examples}

The following system illustrates that, when $V\lambda\neq 0$, it is possible to have $\mathbb{K}=0$ yet $G^*_{s}(v)\cap G^*_{u}(v)=\{0\}$ for some $v\in SM$. In particular, this shows that part (a) of Theorem \ref{theorem:collapse} is optimal. Since $S\mathbb{T}^2$ is a parallelizable manifold, each point $(x,v)\in S\mathbb{T}^2$ can be represented in coordinates $(x,\theta)$, where $x\in [0,1)^2$ and $\theta\in [0, 2\pi)$ is the angle such that $v=\cos(\theta) \partial_{x_1}+\sin(\theta)\partial_{x_2}$.

\begin{proposition} \label{propn:example}
Let $(\mathbb{T}^2, g)$ be the $2$-torus endowed with a flat Riemannian metric. Define $\lambda\in \mathcal{C}^\infty(S\mathbb{T}^2, \R)$ by $\lambda(x, \theta)\coloneqq\cos(\theta)$. 
Then, $\mathbb{K}=0$ and
\begin{equation*}
\begin{alignedat}{3}
&G^*_s(x,\pi/2)= \R(\psi_\lambda - \beta), \qquad && G^*_u(x,\pi/2)=\R\psi_\lambda,\\
& G^*_s(x,-\pi/2)= \R\psi_\lambda, \qquad 
&& G^*_u(x,-\pi/2)= \R(\psi_\lambda - \beta). \\
\end{alignedat}
\end{equation*}
\end{proposition}
\begin{proof}
Note that $(V\lambda)(x, \theta)=-\sin(\theta)$.  Unraveling the definitions, we get
    $$\mathbb{K}=\pi^*K_g - H\lambda + \lambda^2+FV\lambda = \cos^2(\theta)-\cos^2(\theta)=0.$$
Observe that the two tori 
\begin{equation}\label{eq:invariant-tori}
\{(x,\theta)\in S\mathbb{T}^2\mid \theta=\pm \pi/2\}
\end{equation}
are flow-invariant. When $\theta = \pi/2$, the solutions to \eqref{eq:explicit-solutions} become 
$$x(t)=e^{t}x(0), \qquad y(t)=y(0)+x(0)(1-e^{t}).$$
Recall that we had set $p=V\lambda$, so $p(x, \pi/2)=-1$. Further note that $m(t)=e^{t/2}$. If we start with $x(0)= 0$ and $y(0)\neq 0$, which corresponds to a covector in $\R(\psi_\lambda-\beta)$, we get that $z(t)=y(0)e^{-t/2}$ is bounded for all $t\geq 0$. By Lemma \ref{lemma:bounded_implication}, it follows that $G^*_s(x, \pi/2)=\R(\psi_\lambda-\beta)$. On the other hand, if we start with $x(0)\neq 0$ and $y(0)=-x(0)$, which corresponds to a covector in $\R\psi_\lambda$, we get that $z(t)=-x(0)e^{t/2}$ is bounded for all $t\leq 0$. By Lemma \ref{lemma:bounded_implication}, it follows that $G^*_u(x, \pi/2)=\R\psi_\lambda$.

When $\theta=-\pi/2$, we instead have 
$$x(t)=e^{-t}x(0), \qquad y(t)=y(0)+x(0)(e^{-t}-1), \qquad m(t)=e^{-t/2},$$
so we can repeat the same arguments.
\end{proof}

\begin{remark}
This example is also interesting from another perspective. In the classical paper \cite{ballmann87}, the authors construct a surface without
conjugate points where the Green bundles are not continuous. They point out that the Green bundles are always continuous when $K_g \leq 0$. Their example is quite elaborate, but it shares key features with the previous system: our two invariant tori \eqref{eq:invariant-tori} play the role of their closed geodesics $\sigma_\pm$. The advantage of our thermostat example is that it is much easier to understand and visualize (although it is not purely geodesic). 

\end{remark}

To prove Theorem \ref{theorem:counterexample}, we provide the following family of examples.

\begin{figure}[h]
\centering 
\begin{subfigure}[t]
{0.5\textwidth}
\centering
\includegraphics[scale=0.4]{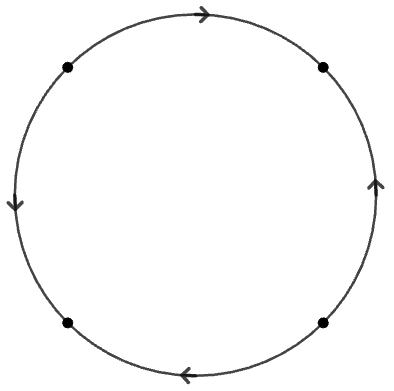}
\caption*{(a) Dynamics restricted to the fibers of $S\mathbb{T}^2$.}
\end{subfigure}%
\begin{subfigure}[t] 
{0.5\textwidth}
\centering
\includegraphics[scale=0.55]{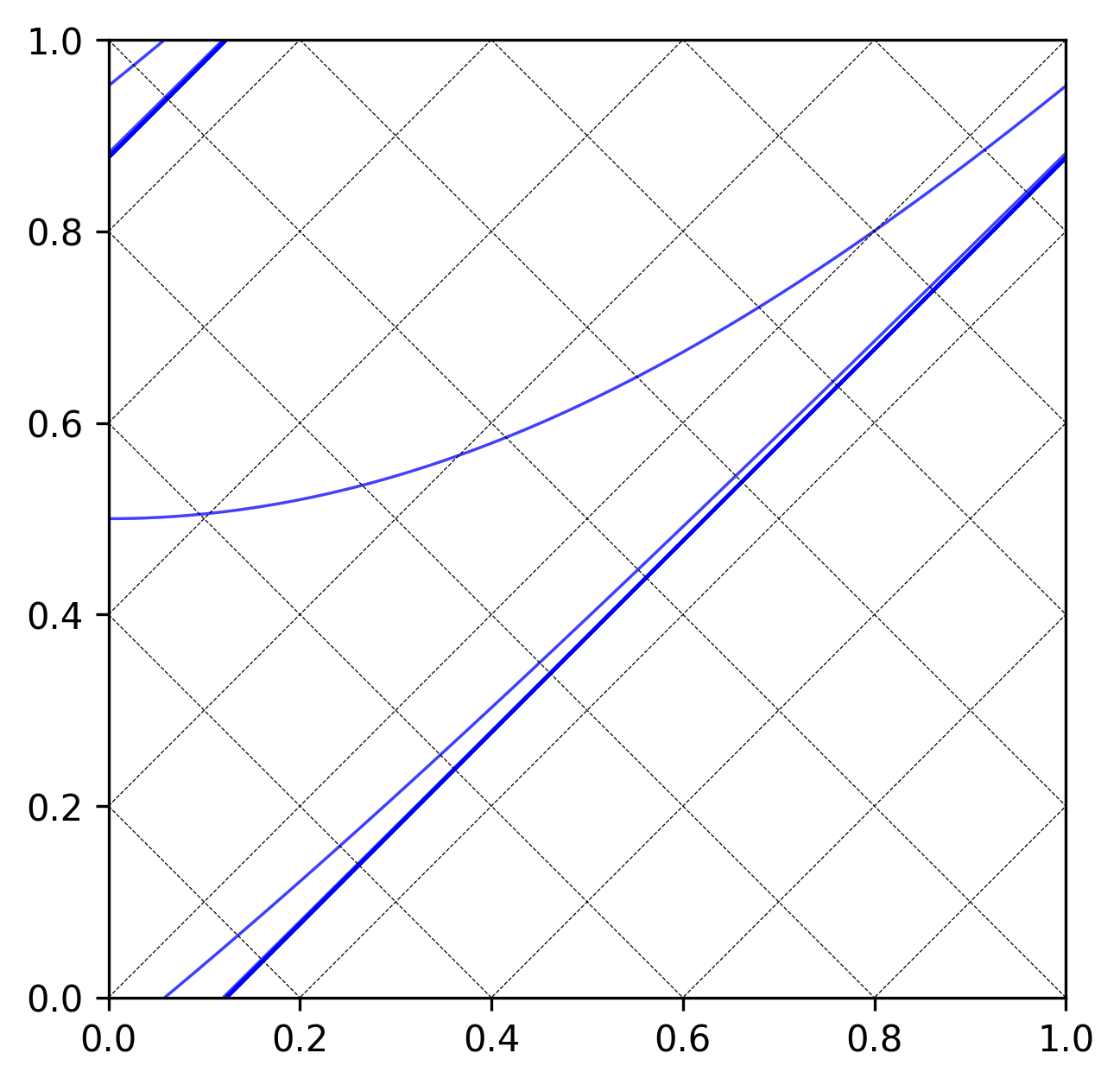}
\caption*{(b) Evolution of a single trajectory on $\mathbb{T}^2$.}
\end{subfigure}
\caption{Illustration when $m=2$.}
\label{figure:m-2}
\end{figure}

\begin{lemma}\label{lem:example-splitting}
Let $g$ be a Riemannian metric on $\mathbb{T}^2$ and define $\lambda\in \mathcal{C}^\infty(S\mathbb{T}^2, \R)$ by $$\lambda(x,\theta)\coloneqq h(x)+\cos(m\theta),$$ where $h\in \mathcal{C}^\infty(\mathbb{T}^2, \R)$ and $m$ is an integer $\geq 2$. If 
%$$\Vert \partial h\Vert_{\infty}+h^2+(m-2)|h|<m-1+(m-2)|h|$$
$$\Vert \partial h\Vert_{\infty}+h^2+(m-2)|h|<m-1-K_g,$$ then the thermostat $(\mathbb{T}^2, g, \lambda)$ has no conjugate points and it is projectively Anosov.

%The conclusion also applies if we take $m=2$ and $\lambda_0=\pi^*h $ for some $h\in \mathcal{C}^\infty(M, \R)$ satisfying $\Vert \partial h\Vert_{\infty}+h^2<1$.
\end{lemma}

\begin{proof}
    Let $p\coloneqq V\lambda/m$, that is, $p(x,\theta)=-\sin(m\theta)$, and let us write $\lambda_0\coloneqq \pi^*h$. Then,
        \begin{equation*}
        \begin{alignedat}{1}
    \kappa_p &= \pi^*K_g-H\lambda+\lambda^2+Fp+p(p-V\lambda)\\
   % &=\pi^*K_g-H(\lambda_0)+\lambda^2-m\lambda\cos(m\theta)+\sin^2(m\theta)-m\sin^2(m\theta)\\
    %&=\pi^*K_g-H(\lambda_0)+\lambda_0^2+\cos^2(m\theta)+(2-m)\lambda_0\cos(m\theta)-m\cos^2(m\theta)+\sin^2(m\theta)-m\sin^2(m\theta)\\
    &=\pi^*K_g-H\lambda_0 + \lambda_0^2+(2-m)\lambda_0\cos(m\theta)+1-m.
        \end{alignedat}
    \end{equation*}
Therefore, since $m\geq 2$, we obtain the inequality
    \begin{equation*}
        \begin{alignedat}{1}
    \kappa_p &\leq \pi^* K_g +\Vert H\lambda_0\Vert_\infty+\lambda_0^2+(m-2)|\lambda_0|+1-m <0.
        \end{alignedat}
    \end{equation*}
The result follows from Theorem \ref{theorem:non-positive-implies-no-conjugate} and \cite[Theorem 3.7]{mettler19}. \end{proof}

By \cite{plante72}, we know that $S\mathbb{T}^2$ cannot admit an Anosov flow and, hence, this is an example of a projectively Anosov thermostat which is not Anosov. It is also easy to see that the non-wandering set is $\Omega=\{(x, \theta)\in S\mathbb{T}^2\mid \cos(m\theta)=0, \, \sin(m\theta)=1\}$. See Figure \ref{figure:m-2}.

\bibliographystyle{alpha}
\bibliography{references}

@article{contreras99,
title={{Convex Hamiltonians without conjugate points}},
author={Contreras, Gonzalo and Iturriaga, Renato},
year={1999},
journal={Ergodic Theory and Dynamical Systems},
volume={19},
number={4},
pages={901-952}
}

@article{mane97,
title={{Lagrangian flows: The dynamics of globally minimizing orbits}},
author={Ma{\~n}{\'e}, Ricardo},
year={1997},
journal={Boletim da Sociedade Brasileira de Matemática},
volume={28},
number={2},
pages={141-153}
}

@article{burns02,
title={{Anosov magnetic flows, critical values and topological entropy}},
author={Burns, Keith and Paternain, Gabriel P.},
year={2002},
journal={Nonlinearity},
volume = {15},
number = {2},
pages = {281-314}
}

@article{contreras04,
title={{Periodic orbits for exact magnetic flows on surfaces}},
author={Contreras, Gonzalo and Macarini, Leonardo and Paternain, Gabriel P.},
year={2004},
journal={International Mathematics Research Notices},
volume={2004},
number={8},
pages={361-387}
}

@article{contreras97,
title={{Lagrangian flows: The dynamics of globally minimizing orbits - II}},
author={Contreras, Gonzalo and Delgado, Jorge and Iturriaga, Renato},
year={1997},
journal={Bulletin of the Brazilian Mathematical Society},
volume={28},
number={2},
pages={155-196}
}

@article{cieliebak10,
title={{Symplectic topology of Mañé's critical values}},
author={Cieliebak, Kai and Frauenfelder, Urs and Paternain, Gabriel P.},
year={2010},
journal={Geometry \& Topology},
volume={14},
number={3},
pages={1765-1870}
}

@article{dairbekov07,
  title={{Entropy production in Gaussian thermostats}},
  author={Dairbekov, Nurlan S. and Paternain, Gabriel P.},
  journal={Communications in Mathematical Physics},
  volume={269},
  number={2},
  pages={533-543},
  year={2007}
}

@article{wojtkowski00,
  title={{Magnetic flows and Gaussian thermostats on manifolds of negative curvature}},
  author={Wojtkowski, Maciej P.},
  journal={Fundamenta Mathematicae},
  volume={163},
  number={2},
  pages={177-191},
  year={2000}
}

@article{wojtkowski00b,
  title={{W-flows on Weyl manifolds and Gaussian thermostats}},
  author={Wojtkowski, Maciej P.},
  journal={Journal de Mathématiques Pures et Appliquées},
  volume={79},
  number={10},
  pages={953-974},
  year={2000}
}

@inproceedings{wojtkowski02,
  author={Wojtkowski, Maciej P.},
  title={{Weyl manifolds and Gaussian thermostats}},
  booktitle={Proceedings of the International Congress of Mathematicians},
  editor={Li, Daqian},
  volume={3},
  pages={511--523},
  publisher={Higher Education Press},
  address={Beijing},
  year={2002}
}

@article{gallavotti97,
title={{SRB states and nonequilibrium statistical mechanics close to equilibrium}},
author={Gallavotti, Giovanni and Ruelle, David},
year={1997},
journal={Communications in Mathematical Physics},
volume={190},
number={2},
pages={279-285}
}

@article{gallavotti99,
title={New methods in nonequilibrium gases and fluids},
author={Gallavotti, Giovanni},
year={1999},
journal={Open Systems \& Information Dynamics},
volume={6},
number={2},
pages={101–136}
}

@article{ruelle99,
title={Smooth dynamics and new theoretical ideas in nonequilibrium statistical mechanics},
author={Ruelle, David},
year={1999},
journal={Journal of Statistical Physics},
volume={95},
number={1-2},
pages={393–468}
}

@article{paternain07,
title={{Regularity of weak foliations for thermostats}},
author={Paternain, Gabriel P.},
year={2007},
journal={Nonlinearity},
volume={20},
number={1},
pages={87-104}
}

@article{ghys92,
title={{Déformations de flots d'Anosov et de groupes fuchsiens}},
author={Ghys, Étienne},
journal={Annales de l'Institut Fourier},
year={1992},
number={1-2},
volume={42},
pages={209-247}
}

@article{ghys93,
title={{Rigidité différentiable des groupes fuchsiens}},
author={Ghys, Étienne},
journal={Publications mathématiques de l’IHÉS},
year={1993},
volume={78},
pages={163-185}
}

@article{mettler20,
title={{Convex projective surfaces with compatible Weyl connection are hyperbolic}},
author={Mettler, Thomas and Gabriel P. Paternain},
journal={Analysis \& PDE},
year={2020},
volume={13},
number={4},
pages={1073-1097}
}

@article{labourie07,
title={{Flat projective structures on surfaces and cubic holomorphic differentials}},
author={Labourie, François},
journal={Pure and Applied Mathematics Quarterly},
year={2007},
volume={3},
number={4},
pages={1057-1099}
}

@article{przytycki08,
title={{Gaussian thermostats as geodesic flows of nonsymmetric linear connections}},
author={Przytycki, Piotr and Wojtkowski, Maciej P.},
journal={Communications in Mathematical Physics},
volume={277},
number={3},
year={2008},
pages={759–769}
}

@book{hoover86,
title={{Molecular Dynamics}},
author={Hoover, William G.},
year={1986},
publisher={Springer},
series={Lecture Notes in Physics},
address={Berlin},
volume={258},
}

@article{cekic26,
title={{Quasi-Fuchsian flows and the coupled vortex equations}}, 
author={Cekić, Mihajlo and Paternain, Gabriel P.},
year={2026},
journal={Bulletin of the London Mathematical Society, \textup{to appear.} Preprint, arXiv:2501.10591},
eprint={2501.10591},
archivePrefix={arXiv},
url={https://arxiv.org/abs/2501.10591}
}

@article{ginzburg50,
title={On the theory of superconductivity},
author={Ginzburg, Vitaly L. and Landau, Lev D.},
journal={Zhurnal Éksperimental'noĭ i Teoreticheskoĭ Fiziki},
volume={20},
year={1950},
pages={1064-1082}
}

@book{jaffe80,
title = {{Vortices and Monopoles: Structure of Static Gauge Theories}},
author={Jaffe, Arthur and Taubes, Clifford},
year={1980},
publisher={Birkhäuser},
series={Progress in Physics},
address={Boston},
volume={2}
}

@article{dumas15,
title = {{Polynomial cubic differentials and convex polygons in the projective plane}},
author={Dumas, David and Wolf, Michael},
year={2015},
journal={Geometric and Functional Analysis},
volume={25},
number={6},
pages={1734-1798 }
}

@incollection{wang91,
title = {{Some examples of complete hyperbolic affine 2-spheres in $\mathbb{R}^3$}},
author={Wang, Changping},
editor={Ferus, Dirk
and Pinkall, Ulrich
and Simon, Udo
and Wegner, Berd},
year={1991},
series={Lecture Notes in Mathematics},
booktitle={Global Differential Geometry and Global Analysis},
volume={1481},
pages={271-280},
address={Berlin},
publisher={Springer}
}

@article{bradlow91,
title={{Special metrics and stability for holomorphic bundles with global sections}},
author={Bradlow, Steven B.},
year={1991},
journal={Journal of Differential Geometry},
volume={33},
number={1},
pages={169-213}
}

@article{garcia-prada94,
title={{A direct existence proof for the vortex equations over a compact Riemann surface}},
author={{García-Prada}, Oscar},
year={1994},
journal={Bulletin of the London Mathematical Society},
volume={26},
number={1},
pages={88-96}
}

@article{witten07,
title={{From superconductors and four-manifolds to weak interactions}},
author={Witten, Edward},
year={2007},
journal={Bulletin of the American Mathematical Society},
volume={44},
number={3},
pages={361-391}
}

@article{noguchi87,
title={{Yang–Mills–Higgs theory on a compact Riemann surface}},
author={Noguchi, Mitsunori},
year={1987},
journal={Journal of Mathematical Physics},
volume={28},
number={10},
pages={2343–2346}
}

@article{ghys84,
title={{Flots d'Anosov sur les 3-variétés fibrées en cercles}},
author={Ghys, Étienne},
journal={Ergodic Theory and Dynamical Systems},
year={1984},
pages={67-80},
volume={4},
number={1}
}

@article{assylbekov14,
  title={Hopf type rigidity for thermostats},
  author={Assylbekov, Yernat M. and Dairbekov, Nurlan S.},
  journal={Ergodic Theory and Dynamical Systems},
  volume={34},
  number={6},
  pages={1761-1769},
  year={2014}
}

@article{eberlein73,
title={{When is a geodesic flow of Anosov type? I}},
author={Eberlein, Patrick},
journal={Journal of Differential Geometry},
year={1973},
volume={8},
number={3},
pages={437-463}
}

@article{mettler19,
title={{Holomorphic differentials, thermostats and Anosov flows}},
author={Mettler, Thomas and Paternain, Gabriel P.},
year={2019},
journal={Mathematische Annalen},
volume={373},
pages={553-580},
number={1}
}

@article{ballmann87,
title={{On surfaces with no conjugate points}},
author={Blallmann, Werner and Brin, Michael and Burns, Keith},
year={1987},
journal={Journal of Differential Geometry},
volume={25},
pages={249-273}
}

@article{freire82,
  title={{On the entropy of the geodesic flow in manifolds without conjugate points}},
  author={Freire, Alexandre and Mañé, Ricardo},
  journal={Inventiones mathematicae},
  volume={69},
  number={3},
  pages={375-392},
  year={1982}
}

@article{green54,
  title={{Surfaces without conjugate points}},
  author={Green, Leon W.},
  journal={Transactions of the American Mathematical society},
  volume={76},
  number={3},
  pages={529-546},
  year={1954}
}

@article{klingenberg74,
    title={{Riemannian manifolds with geodesic flow of Anosov type}},
    author={Klingenberg, Wilhelm},
    year={1974},
    journal={Annals of Mathematics},
    pages={1-13},
    volume={99},
    number={1}
}

@book{barreira02,
  title={{Lyapunov Exponents and Smooth Ergodic Theory}},
  author={Barreira, Luís and Pesin, Yakov B.},
  volume={23},
  year={2002},
  series={University Lecture Series},
  publisher={American Mathematical Society},
  address={Providence}
}

@article{dairbekov07b,
title={{The boundary rigidity problem in the presence of a magnetic field}},
author={Dairbekov, Nurlan S. and Paternain, Gabriel P. and Stefanov, Plamen and Uhlmann, Gunther},
volume={216},
number={2},
year={2007},
journal={Advances in Mathematics},
pages={535--609}
}

@article{osedelets68,
  title={{A multiplicative ergodic theorem. Characteristic Ljapunov, exponents of dynamical systems}},
  author={Osedelets, Valeriy I.},
  journal={Trudy Moskovskogo Matematicheskogo Obshchestva},
  volume={19},
  pages={179-210},
  year={1968}
}

@book{fisher19,
  title={{Hyperbolic Flows}},
  author={Fisher, Todd and Hasselblatt, Boris},
  volume={25},
  series={Zurich Lectures in Advanced Mathematics},
  publisher={European Mathematical Society},
  year={2019},
  address={Berlin}
}

@article{hopf48,
    title={{Closed surfaces without conjugate points}},
    author={Hopf, Eberhard},
    volume={34},
    number={2},
    pages={47-51},
    year={1948},
    journal={Proceedings of the National Academy of Sciences of the United States of America}
}

@article{plante72,
  title={{Anosov flows and the fundamental group}},
  author={Plante, Joseph F. and Thurston, William P.},
  journal={Topology},
  volume={11},
  number={2},
  pages={147-150},
  year={1972}
}

@article{anosov67,
title={{Geodesic flows on closed Riemann manifolds with negative curvature}},
author={Anosov, Dmitri V.},
year={1967},
journal={Trudy Matematicheskogo Instituta imeni V. A. Steklova},
volume={90},
pages={3-210}
}

@article{assenza25,
title={{Magnetic flatness and E. Hopf’s theorem for magnetic systems}},
author={Assenza, Valerio and {Marshall Reber}, James and Terek, Ivo},
year={2025},
journal={Communications in Mathematical Physics},
volume={406},
number={2}
}

@article{arroyo03,
  title={{Homoclinic bifurcations and uniform hyperbolicity for three-dimensional flows}},
  author={Arroyo, Aubin and {Rodríguez Hertz}, Federico},
  journal={Annales de l'Institut Henri Poincaré C, Analyse non linéaire},
  volume={20},
  number={5},
  pages={805-841},
  year={2003}
}

@article{ruelle78,
  title={{An inequality for the entropy of differentiable maps}},
  author={Ruelle, David},
  journal={Boletim da Sociedade Brasileira de Matemática},
  volume={9},
  number={1},
  pages={83-87},
  year={1978}
}

@article{jane09,
title={{On the injectivity of the X-ray transform for Anosov thermostats}},
author={Jane, Dan and Paternain, Gabriel P.},
year={2009},
journal={Discrete and Continuous Dynamical Systems},
volume={24},
number={2},
pages={471–487}
}

@article{merry11,
title={{Inverse problems in geometry and dynamics}},
journal={Lecture notes, \textup{University of Cambridge}},
author={Merry, Will J. and Paternain, Gabriel P.},
year={2011}
}

@book{araujo10,
title={{Three-Dimensional Flows}},
author={Araújo, Vítor and Pacifico, Maria José},
year={2010},
series={A Series of Modern Surveys in Mathematics},
volume={53},
publisher={Springer},
address={Berlin}
}

@article{blair98,
title={{Conformally Anosov flows in contact metric geometry}},
author={Blair, David E. and Peronne, Domenico},
year={1998},
journal={Balkan Journal of Geometry and Its Applications},
volume={3},
number={2},
pages={33--46}
}

@book{eliashberg98,
title={{Confoliations}},
author={Eliashberg, Yakov M. and Thurston, William P.},
year={1998},
series={University Lecture Series},
volume={13},
publisher={American Mathematical Society},
address={Providence}
}

@incollection{donnay03,
title = {{Anosov geodesic flows for embedded surfaces}},
author={Donnay, Victor J. and Pugh, Charles C.},
year={2003},
editor = {de Melo, Wellington and Viana, Marcelo and Yoccoz, Jean-Christophe},
series={Astérisque},
booktitle={Geometric Methods in Dynamics (II): Volume in Honor of Jacob Palis},
volume={287},
pages={61-69},
address={Paris},
publisher={Société Mathématique de France}
}

@article{donnay18,
title={{A new proof of the existence of embedded surfaces with Anosov geodesic flow}},
author={Donnay, Victor J. and Visscher, Daniel},
year={2018},
journal={Regular and Chaotic Dynamics},
volume={23},
number={6},
pages={685--694}
}

\end{document}